\documentclass[10pt,a4paper]{amsart}
\usepackage[T1]{fontenc}
\usepackage[utf8]{inputenc}
\usepackage{amsmath,amssymb,mathtools,mathrsfs}
\usepackage[textwidth=16cm,textheight=22cm,centering]{geometry}
\usepackage[shortlabels]{enumitem}
\usepackage{aliascnt,etoolbox}
\usepackage{needspace}
\makeatletter
\patchcmd{\@setaddresses}{\/:\space}{\/\space}{}{}
\patchcmd{\@setaddresses}{\/:\space}{\/\space}{}{}
\patchcmd{\@setaddresses}{\/:\space}{\/\space}{}{}
\makeatother
\usepackage{xcolor}
\usepackage[nocompress]{cite}
\usepackage[colorlinks=true,citecolor=red,linkcolor=blue,urlcolor=blue]{hyperref}
\theoremstyle{plain}
\newtheorem{thm}{Theorem}[section]
\newaliascnt{prop}{thm}
\newtheorem{prop}[prop]{Proposition}
\aliascntresetthe{prop}
\newaliascnt{lem}{thm}
\newtheorem{lem}[lem]{Lemma}
\aliascntresetthe{lem}
\theoremstyle{definition}
\newaliascnt{defn}{thm}

\aliascntresetthe{defn}
\newaliascnt{asu}{thm}
\newtheorem{asu}[asu]{Assumption}
\aliascntresetthe{asu}

\newaliascnt{cor}{thm}
\newtheorem{cor}[cor]{Corollary}
\aliascntresetthe{cor}
\newaliascnt{rem}{thm}
\newtheorem{rem}[rem]{Remark}
\aliascntresetthe{rem}

\newcounter{proofnumber}
\newcounter{stp}
\theoremstyle{definition}

\AtBeginEnvironment{proof}{\stepcounter{proofnumber}\setcounter{stp}{0}}
\numberwithin{equation}{section}
\setlist[enumerate]{font=\normalfont,label=(\roman*),leftmargin=*}
\allowdisplaybreaks[2]
\newcommand{\R}{\mathbb R}
\newcommand{\T}{\mathbb T}
\newcommand{\N}{\mathbb N}
\newcommand{\rL}{\mathrm L}
\newcommand{\rH}{\mathrm H}
\newcommand{\rC}{\mathrm C}
\newcommand{\rW}{\mathrm W}
\newcommand{\rX}{\mathrm X}

\newcommand{\rV}{\mathrm V}
\newcommand{\rZ}{\mathrm Z}
\newcommand{\dt}{\partial_t}
\newcommand{\dz}{\partial_z}
\newcommand{\dd}{\,\mathrm d}
\newcommand{\divH}{\operatorname{div}_{\rH}}
\newcommand{\nablaH}{\nabla_{\rH}}
\newcommand{\DeltaH}{\Delta_{\rH}}

\title[Compressible primitive equations with physical vacuum]{Local smooth solutions of the free-boundary compressible primitive equations with physical vacuum}
\author{Tarek Z\"ochling}
\address{Technische Universit\"at Darmstadt, Schlo\ss{}gartenstra\ss{}e 7, 64289 Darmstadt, Germany}
\email{zoechling@mathematik.tu-darmstadt.de}
\date{}

\keywords{Compressible primitive equations, free boundary,
physical vacuum, local existence, weighted energy estimates,
Nash--Moser iteration \\
Tarek Z\"ochling gratefully acknowledge the support by the Deutsche Forschungsgemeinschaft (DFG) through the Research Unit FOR~5528}

\hypersetup{pdftitle={Local smooth solutions of the free-boundary
compressible primitive equations with physical vacuum}}
\begin{document}
\begin{abstract}
Local existence and uniqueness of smooth solutions are established
for the three-dimensional compressible primitive equations with
constant viscosity and a physical vacuum free boundary.
We employ a Nash-Moser approach to address two distinct derivative
losses, one arising from the free boundary geometry in the
transformed viscosity and the other from the diagnostic
reconstruction of the vertical velocity. We handle the first
through Alinhac's good velocity unknown and the second through
a modified height. Combining these two corrections with weighted
energy estimates adapted to the vacuum degeneracy yields the
tame estimates required for the iteration.
\end{abstract}
\maketitle

\section{Introduction}

Large-scale atmospheric flows are commonly described using the
hydrostatic approximation, reflecting the separation between their
horizontal and vertical length scales. In this approximation, the
vertical pressure gradient balances gravity, so that
\begin{equation*}
 \partial_z p=-g\varrho,
\end{equation*}
where $\varrho$ denotes the density.
This balance is a fundamental ingredient of atmospheric and oceanic
models, see, for example, \cite{LTW-atmosphere-92,LTW-ocean-92}.
In particular, the pressure decreases with height wherever the density
is positive.

When the polytropic pressure law $p=\varrho^\gamma$ with
$\gamma>1$ is imposed, hydrostatic balance implies that
$\varrho^{\gamma-1}$ decreases linearly with height and,
starting from a finite positive density, reaches zero at a finite
height. Both density and pressure vanish there, and extending
them by zero above this height connects the atmosphere to vacuum.
In this model, the resulting surface represents an idealized
boundary between the atmosphere and the surrounding space,
as in \cite{LTX}. Its location evolves with the flow, leading
to a free boundary value problem for the compressible primitive
equations in which the upper boundary is unknown and both
pressure and density vanish there. This problem is formulated
in \autoref{sec:formulation}. 

Our main result, \autoref{thm:vacuum-existence}, establishes local
existence and uniqueness of smooth solutions for this model.
We construct the solution by a Nash-Moser scheme applied to the full
transformed equations. This scheme addresses an actual derivative loss
in the direct Sobolev estimates. The transformed viscosity contains
second derivatives of the height, although its continuity equation
provides no parabolic smoothing. At each linearized correction step,
we use the good-unknown mechanism introduced by Alinhac \cite{Alinhac}
to cancel the highest derivatives of the height correction. This
geometric cancellation was also employed by Masmoudi and Rousset
\cite{MR} in their analysis of the viscous free-surface Navier-Stokes
equations. However, in the hydrostatic problem, reconstructing the vertical
velocity from the continuity equation costs one horizontal
derivative. We address this additional derivative loss by
introducing a second good unknown, defined by modifying the
height through the elliptic inverse.
The modified height removes the leading velocity coupling
from the height equation. The toy model in
\autoref{sec:toy-model} explains the two corrections and the basic
energy estimate.

The retention of the complete viscosity distinguishes this problem from
the model of Liu, Titi, and Xin \cite{LTX}. There, the transformed
viscosity and boundary conditions are linearized around equilibrium,
geometric terms that are at least quadratic in the perturbation are
discarded, and a correction is added to restore momentum conservation,
as explained in \cite[Remarks~1-2]{LTX}. The global small-perturbation
theorem concerns this modified model. Here all contributions of the
original physical viscosity and traction condition are retained, and
their derivative loss is handled in the linearized estimates and the
Nash-Moser iteration.

Local existence of smooth solutions has a long history in free-boundary
fluid mechanics. In particular, Lindblad's work on the incompressible
Euler equations \cite{LindbladIncompressible}, constructs local smooth solutions under
the Taylor sign condition using tame estimates and Nash-Moser
iteration. The corresponding treatment of compressible liquids
\cite{Lindblad} uses the same type of iteration, with strictly positive
density at the free surface. For irrotational water waves, tame estimates
for the linearized equations and Nash-Moser iteration also yield local
well-posedness in \cite{Lannes}. For compressible gases with a physical
vacuum, local well-posedness was established using degenerate parabolic
regularization in \cite{CoutandShkoller} and weighted energy estimates
with a linear approximation scheme in \cite{JangMasmoudi}.
More recently, an Eulerian theory in lower-regularity Sobolev spaces,
including continuous dependence on the initial data, was developed in
\cite{IfrimTataru}. 

Within the hydrostatic approximation, local well-posedness for the
inviscid incompressible primitive equations of an ocean with a moving upper surface
and analytic initial data was proved in \cite{IKZ}. 
For viscous free-surface flows without the hydrostatic approximation,
a classical local theory for the incompressible Navier-Stokes equations
was established in \cite{Beale}. Constant viscosity in the presence of
a physical vacuum was treated for the compressible Navier-Stokes-Poisson
system in spherical symmetry in \cite{Jang}. More directly related to
the vacuum degeneracy here, Gui, Wang, and Wang \cite{GWW} established
local well-posedness for the full three-dimensional compressible
Navier-Stokes equations with constant viscosity and a vacuum free
surface. 

For the incompressible primitive equations on fixed domains, the
existence theory is substantially more developed. Foundational results
on global weak solutions were obtained in
\cite{LTW-atmosphere-92,LTW-ocean-92}. Cao and Titi \cite{CaoTiti}
subsequently proved global existence and uniqueness of strong solutions
in three dimensions for arbitrarily large $\rH^1$ initial data. The $\rL^p$ approach in \cite{HK}
extended global strong well-posedness to initial data with lower
differentiability by means of the hydrostatic Stokes operator.

The compressible primitive equations remain less understood.
Local strong well-posedness with gravity and positive density, and also
with vacuum when gravity is absent, was established in \cite{LT-21}.
The zero Mach number limit for well-prepared initial data was justified
in \cite{LT-Mach}, together with global strong solutions close to
incompressible flows under the assumptions of that theorem.
For density-dependent viscosities that vanish at vacuum, global weak
solutions without gravity were constructed in \cite{LT-weak}.
These results concern distinct density and viscosity regimes and do
not provide a theory for the full moving physical-vacuum problem with
constant viscosity studied here.

In \cite{HIRZ}, we developed a hydrostatic Lagrangian approach with
maximal regularity, obtaining local strong solutions for positive-density
data and global solutions for small perturbations of isothermal
equilibria, including gravity. For the two-dimensional isothermal
problem without gravity, we established global strong well-posedness
for arbitrarily large data in \cite{HJMZ}. Beyond barotropic pressure
laws, we treated the heat-conducting compressible primitive equations
with gravity and proved global strong well-posedness near equilibrium
in \cite{ZoechlingHeat}. The free boundary and the vanishing density
in the present problem require the weighted estimates and tame
iteration developed below.

The article is organized as follows. In \autoref{sec:formulation}, we
state the physical problem and reduce it using hydrostatic balance.
The compatibility conditions and main theorem are given in
\autoref{sec:main-result}. In \autoref{sec:transformation}, we transform
the full system to the fixed domain and identify the height derivatives
in the viscosity. \autoref{sec:nash-moser} begins with the toy model,
then derives the complete linearized equations and establishes the
basic and higher-order estimates, recovery of normal derivatives and
construction of the smooth tame linear inverse. Finally,
\autoref{sec:nonlinear-existence} proves nonlinear existence and
uniqueness by the Nash-Moser inverse theorem.

\section{Setting of the problem}
\label{sec:formulation}

We consider the problem of a hydrostatic compressible atmosphere
connected to vacuum by a free boundary. This is described by the
compressible primitive equations in a moving domain. We write
$x_\rH=(x_1,x_2)\in\T^2$ for the horizontal coordinates and $z$ for
the vertical coordinate. For the unknown height $H=H(x_\rH,t)>0$,
we define the atmospheric domain by
\begin{equation*}
 \Omega_H(t)=\{(x_\rH,z)\mid x_\rH\in\T^2,\ 0<z<H(x_\rH,t)\}.
\end{equation*}
The unknowns are the density $\varrho>0$, the velocity
$u=(v,w)$ with
$v=(v_1,v_2)$, and the pressure $p$. We assume the polytropic law
$p=\varrho^\gamma$, where $\gamma>1$, and constant coefficients
$g>0$ and $\mu,\mu'>0$. To specify the equations, we define
\begin{equation*}
 \nablaH=(\partial_{x_1},\partial_{x_2})^\top,\quad
 \divH v=\partial_{x_1}v_1+\partial_{x_2}v_2,\quad
 \DeltaH=\partial_{x_1}^2+\partial_{x_2}^2
 \ \text{ and }\ \Delta=\DeltaH+\partial_z^2.
\end{equation*}
In the systems below, the notation $(0,T)\times\Omega_H(t)$
means the moving space-time region with $0<t<T$ and
$(x_\rH,z)\in\Omega_H(t)$. With the convention
$(\nablaH v)_{ij}=\partial_{x_j}v_i$, the full problem reads
\begin{equation}\label{eq:physical-system}
 \left\{\begin{aligned}
 \dt\varrho+\divH(\varrho v)+\dz(\varrho w)&=0
       &&\text{ in }(0,T)\times\Omega_H(t),\\
 \varrho(\dt v+u\cdot\nabla v)-\mu\Delta v-\mu'\nablaH\divH v+\nablaH p&=0
       &&\text{ in }(0,T)\times\Omega_H(t),\\
 \dz p&=-g\varrho
       &&\text{ in }(0,T)\times\Omega_H(t),\\
 p&=\varrho^\gamma
       &&\text{ in }(0,T)\times\Omega_H(t),\\
 H(0)=H_0,\quad\varrho(0)=\varrho_0
       \ \text{ and }\ v(0)&=v_0.
 \end{aligned}\right.
\end{equation}
At the free surface and at the fixed bottom, respectively, the
boundary conditions for $t\in(0,T)$ and $x_\rH\in\T^2$ read
\begin{equation}\label{eq:physical-boundary}
 \left\{
 \begin{aligned}
 p|_{z=H}=0\ \text{ and }\ w|_{z=H}&=\dt H+v|_{z=H}\cdot\nablaH H,\\
 \left.\big(\mu\dz v-
 \bigl[\mu(\nablaH v+(\nablaH v)^\top)
           +(\mu'-\mu)(\divH v)I_2\bigr]\nablaH H\big)\right|_{z=H}&=0,\\
 w|_{z=0}=0\ \text{ and }\ \dz v|_{z=0}&=0.
 \end{aligned}\right.
\end{equation}
These prescribe zero exterior pressure, a material free surface,
zero viscous traction, and an impermeable slip bottom.
All fields are horizontally periodic. 

We now reformulate the above problem using the explicit vertical dependence from the hydrostatic balance as well as the polytropic pressure law. From now on, we restrict to $\gamma=2$ and set $g=1$ for notational convenience.
Combining the hydrostatic balance with the polytropic pressure law
in \eqref{eq:physical-system} gives
\begin{equation*}
 \partial_z(2\varrho+z)=0.
\end{equation*}
Integrating in the vertical direction and using the vacuum condition
$\varrho|_{z=H}=0$ yields $2\varrho=H-z$.
Consequently, the density, pressure and horizontal pressure gradient
are given by
\begin{equation}\label{eq:hydrostatic-density}
 \varrho=\frac{H-z}{2},\quad p=\frac{(H-z)^2}{4}
 \ \text{ and }\ \nablaH p=\varrho\nablaH H.
\end{equation}

To obtain an evolution equation for $H$, we integrate continuity
in the vertical direction. Using the above relations and
the boundary conditions for $w$, we obtain
\begin{equation}\label{eq:physical-height}
 \frac H2\dt H+\divH\int_0^H\varrho v\,\dd z=0.
\end{equation}
Once the density is represented by \eqref{eq:hydrostatic-density},
the unaveraged continuity equation determines the vertical transport
diagnostically. More precisely, integrating from the bottom and
substituting \eqref{eq:physical-height} gives
\begin{equation}\label{eq:physical-diagnostic}
 (H-z)w=\frac zH\divH\int_0^H(H-s)v(x_\rH,s,t)\,\dd s
 -\int_0^z\divH\bigl((H-s)v(x_\rH,s,t)\bigr)\,\dd s.
\end{equation}
Thus, the right-hand side depends only on $(H,v)$ and contains one
horizontal derivative of $v$. 
The boundary conditions for $w$ are automatically satisfied by this representation. Indeed, at the bottom, \eqref{eq:physical-diagnostic} gives
$H w|_{z=0}=0$, so $w|_{z=0}=0$ because $H>0$.
At the free surface, it holds that
\begin{equation*}
 \begin{aligned}
 w|_{z=H}
 &=\lim_{z\uparrow H}
 \frac{\displaystyle
       \frac zH\divH\int_0^H(H-s)v(x_\rH,s,t)\,\dd s
       -\int_0^z\divH\bigl((H-s)v(x_\rH,s,t)\bigr)\,\dd s}
      {H-z}\\
 &=-\frac1H\divH\int_0^H(H-s)v(x_\rH,s,t)\,\dd s
       +v|_{z=H}\cdot\nablaH H\\
 &=\dt H+v|_{z=H}\cdot\nablaH H.
 \end{aligned}
\end{equation*}
Consequently, bottom impermeability and the material free-surface
condition are already encoded in the weighted diagnostic and the
height equation.
With these preparations, the system \eqref{eq:physical-system}
can be recast as
\begin{equation}\label{eq:physical-reduced}
 \left\{\begin{aligned}
 \dt H+\frac1H\divH\int_0^H(H-z)v\,\dd z&=0
       &&\text{ in }(0,T)\times\T^2,\\
 \frac{H-z}{2}
       (\dt v+v\cdot\nablaH v+w\dz v+\nablaH H)
       -\mu\Delta v-\mu'\nablaH\divH v&=0
       &&\text{ in }(0,T)\times\Omega_H(t),\\
 H(0)=H_0\ \text{ and }\ v(0)&=v_0.
 \end{aligned}\right.
\end{equation}
Its boundary conditions for $t\in(0,T)$ and $x_\rH\in\T^2$ are
\begin{equation}\label{eq:physical-reduced-boundary}
 \left\{\begin{aligned}
 \left.\big(\mu\dz v-
 \bigl[\mu(\nablaH v+(\nablaH v)^\top)
      +(\mu'-\mu)(\divH v)I_2\bigr]\nablaH H\big)\right|_{z=H}&=0,\\
 \dz v|_{z=0}&=0.
 \end{aligned}\right.
\end{equation}
In this formulation, the prognostic variables are the height $H$
and the horizontal velocity $v$. The remaining quantities are
determined diagnostically through
$\varrho=(H-z)/2$, $p=(H-z)^2/4$ and the weighted representation
\eqref{eq:physical-diagnostic}. In particular,
$\varrho_0=(H_0-z)/2$. Since the vertical transport in
\eqref{eq:physical-reduced} involves only the product $(H-z)w$,
substituting \eqref{eq:physical-diagnostic} gives a closed system
for $(H,v)$.
\section{Main result}\label{sec:main-result}
Since $H-z$ vanishes at the upper boundary, the
initial data must satisfy compatibility conditions at every order.
These compatibility conditions are formulated in the following.
\begin{asu}[Smooth compatibility]\label{def:vacuum-compatibility}
Let $H_0\in\rC^\infty(\T^2)$ be strictly positive and let
$v_0\in\rC^\infty(\overline{\Omega_{H_0}})^2$. We define the formal initial time derivatives recursively, starting
from $H|_{t=0}=H_0$ and $v|_{t=0}=v_0$. 
The first derivatives
are obtained by evaluating the right-hand sides of
\eqref{eq:physical-reduced} at $t=0$, and higher derivatives $\dt^{j+1} H |_{t=0}$ and $\dt^{j+1} v |_{t=0}$ are
obtained by differentiating these expressions $j$-times in time and evaluating at $t=0$. Setting
\begin{equation*}
     H_j=\partial_t^jH|_{t=0}
 \ \text{ and }\
 v_j=\partial_t^jv|_{t=0},
\end{equation*}
we assume that the initial data satisfy
\begin{equation*}
 H_j\in \rC^\infty(\T^2) , \quad \inf H_0 >0
 \ \text{ and }\
 v_j\in \rC^\infty(\overline{\Omega_{H_0}})^2 \ \text{ for } j \in \N_0.
\end{equation*}
Moreover assume that 
the boundary conditions hold at every time derivative order,
which means that
\begin{equation*}
 \begin{aligned}
 \partial_t^j(\partial_zv|_{z=0})
 |_{t=0}=0 \ \text{ and } \ 
 \partial_t^j\big(
 \left[\mu\partial_zv
 -\bigl(\mu(\nablaH v+(\nablaH v)^\top)
       +(\mu'-\mu)(\divH v)I_2\bigr)\nablaH H
 \right]_{z=H}
 \big)|_{t=0}=0.
 \end{aligned}
\end{equation*}
\end{asu}

The smoothness of the initial acceleration required in
\autoref{def:vacuum-compatibility} implies that
\begin{equation*}
 \left.(\mu\Delta v_0+\mu'\nablaH\divH v_0)\right|_{z=H_0}=0.
\end{equation*}
Indeed, this follows by evaluating the momentum equation at
$t=0$ and $z=H_0$, since both $H_0-z$ and the weighted vertical transport vanish
at the free surface. We are now in a position to state the main theorem of this article.

\begin{thm}\label{thm:vacuum-existence}
Suppose that $H_0$ and $v_0$ satisfy
\autoref{def:vacuum-compatibility}. Then there exist $T>0$
and a unique pair $(H,v)$ solving \eqref{eq:physical-reduced}
subject to \eqref{eq:physical-reduced-boundary}, with the product
$(H-z)w$ given by \eqref{eq:physical-diagnostic}, such that
\begin{equation*}
 H\in\rC^\infty([0,T]\times\T^2)
 \quad\text{with } H(x_\rH,t)\geq\tfrac12\min_{\T^2}H_0
 \ \text{ and }\
 v\in\rC^\infty([0,T]\times\overline{\Omega_H(t)})^2.
\end{equation*}
\end{thm}

\begin{rem}
The compatibility class contains nonzero data. For a constant height
$H_0=\overline H>0$, take $v_0=(a(x_2)b(z),0)$, where
$a\in\rC^\infty(\T)$ and $b\in\rC_c^\infty(0,\overline H)$.
The recursion then gives
\begin{equation*}
 H_{j+1}=0\ \text{ and }\ 
 v_{j+1}=\frac{2\mu}{\overline H-z}
       (\partial_{x_2}^2+\partial_z^2)v_j\quad(j\geq0).
\end{equation*}
Every finite iterate is smooth and supported away from both boundaries,
so all compatibility conditions hold.
\end{rem}

\section{Transformation to the fixed domain}
\label{sec:transformation}
We write $\rL^p$, $\rH^s$, $\rW^{s,p}$ and $\rC^k$ for the usual
Lebesgue, Sobolev and classical differentiability spaces on the
indicated domain. 
We define the reference
domain $\Omega$ by $\Omega=\T^2\times(0,1)$. For a smooth function $f$ on $\T^2$, we define
its harmonic extension $Ef$ by
\begin{equation*}
 (\DeltaH+\partial_z^2)Ef=0\quad\text{in }\Omega,\quad
 Ef|_{z=0}=0\ \text{ and }\ Ef|_{z=1}=f.
\end{equation*}
Equivalently, with the convention $\T=\R/\mathbb Z$, we obtain
\begin{equation*}
 Ef(x_\rH,z)=z\widehat f(0)
 +\sum_{k\in\mathbb Z^2\setminus\{0\}}
       \frac{\sinh(2\pi|k|z)}{\sinh(2\pi|k|)}
          \widehat f(k)e^{2\pi i k\cdot x_\rH},
\end{equation*}
so that 
\begin{equation*}
 \|Ef\|_{\rH^1(\Omega)}
 \leq C\|f\|_{\rH^{1/2}(\T^2)},
\end{equation*}
Here $C$ depends only on $\Omega$.
For a time-dependent function this problem is solved at each fixed
time. We then define the vertical coordinate $\zeta$ and the map
$\phi$ by
\begin{equation*}
 \zeta(x,t)=zH_0(x_\rH)+E(H-H_0)(x,t)
 \ \text{ and }\ \phi(x,t)=(x_\rH,\zeta(x,t)).
\end{equation*}
Thus $\zeta|_{z=0}=0$, $\zeta|_{z=1}=H$, and
$\zeta(x,0)=zH_0(x_\rH)$. To check that this is a valid change of
coordinates, we define $J=\det D_x\phi$ and calculate
\begin{equation*}
 D_x\phi=
 \begin{pmatrix}
  1&0&0\\0&1&0\\
  \partial_{x_1}\zeta&\partial_{x_2}\zeta&H_0+\partial_zE(H-H_0)
 \end{pmatrix}
 \ \text{ and }\ J=H_0+\partial_zE(H-H_0)=\partial_z\zeta.
\end{equation*}
Writing $H_\ast=\min_{\T^2}H_0>0$, we choose $T>0$ sufficiently
small so that
\begin{equation*}
 \sup_{0\leq t\leq T}
 \|\partial_zE(H(t)-H_0)\|_{\rL^\infty(\Omega)}
 <\frac{H_\ast}{2}.
\end{equation*}
It then follows that
\begin{equation*}
 J=H_0+\partial_zE(H-H_0)\geq\frac{H_\ast}{2}>0
 \ \text{ and }\
 H=\int_0^1J\,\dd z\geq\frac{H_\ast}{2}
 \quad\text{for }t\in[0,T].
\end{equation*}
For each $x_\rH$, the function $z\mapsto\zeta(x_\rH,z,t)$ is
strictly increasing from $0$ to $H(x_\rH,t)$. Consequently,
$\phi(\cdot,t)$ is a smooth diffeomorphism from $\overline\Omega$
to $\overline{\Omega_H(t)}$. 
To express the reduced problem in these coordinates, we define the
transformed horizontal velocity $V$, pressure $P$ and density $R$ by
\begin{equation*}
 V(x,t)=v(\phi(x,t),t),\quad P(x,t)=p(\phi(x,t),t)
 \ \text{ and }\ R(x,t)=\varrho(\phi(x,t),t).
\end{equation*}
In particular, \eqref{eq:hydrostatic-density} gives
\begin{equation*}
 R=\frac{H-\zeta}{2},\quad P=\frac{(H-\zeta)^2}{4},\quad
 Q=JR=\frac{J(H-\zeta)}2
 \ \text{ and }\ Q_0=Q|_{t=0}=\frac{H_0^2}{2}(1-z).
\end{equation*}
The Jacobian bound also gives the smooth positive quotient
\begin{equation*}
 \frac Q{1-z}=\frac J2\int_0^1
 J(x_\rH,1-\theta(1-z),t)\,\dd\theta
 \geq\frac{H_\ast^2}{8}.
\end{equation*}
Setting
\begin{equation*}
 D_i^\phi=\partial_{x_i}-\frac{\partial_{x_i}\zeta}{J}\partial_z
 \quad(i=1,2),\quad D_3^\phi=J^{-1}\partial_z,\quad
 \nablaH^\phi=(D_1^\phi,D_2^\phi)^\top
 \ \text{ and }\ \divH^\phi V=D_1^\phi V_1+D_2^\phi V_2
\end{equation*}
one calculates
$(\partial_{x_i}v)\circ\phi=D_i^\phi V$,
$(\partial_zv)\circ\phi=D_3^\phi V$, and
$(\partial_tv)\circ\phi=\partial_tV-(E\partial_tH)D_3^\phi V$.
With $\Delta=\DeltaH+\partial_z^2$ and
$V_0(x_\rH,z)=v_0(x_\rH,zH_0(x_\rH))$, the transformed system reads as
\begin{equation}\label{eq:reduced-problem}
 \left\{\begin{aligned}
 \dt H+\frac2{H_0}\divH\int_0^1Q_0V\,\dd z&=f_1(H,V)
      &&\text{ in }(0,T)\times\T^2,\\
 Q_0(\dt V+\nablaH H)-\mu\Delta V-\mu'\nablaH\divH V&=f_2(H,V)
      &&\text{ in }(0,T)\times\Omega,\\
 H(0)=H_0\ \text{ and }\ V(0)&=V_0,
 \end{aligned}\right.
\end{equation}
where the nonlinear terms $f_1(H,V)$ and $f_2(H,V)$ are given by 
\begin{equation}\label{eq:transformed-f1}
 f_1=\frac{2(H-H_0)}{HH_0}\divH\int_0^1Q_0V\,\dd z
       -\frac2H\divH\int_0^1(Q-Q_0)V\,\dd z.
\end{equation}
as well as 
\begin{equation}\label{eq:transformed-f2}
 \begin{aligned}
 f_2={}&(Q_0-Q)(V_t+\nablaH H)
       -Q V\cdot\nablaH V-\mathcal B V_z\\
 &+\mu\divH\bigl((J-1)\nablaH V-V_z\otimes\nablaH\zeta\bigr)
  +\mu\dz\left[\left(\frac{1+|\nablaH\zeta|^2}{J}-1\right)V_z
                         -(\nablaH V)\nablaH\zeta\right]\\
 &+\mu'\nablaH\bigl((J-1)\divH V-\nablaH\zeta\cdot V_z\bigr)
  -\mu'\dz\left[\nablaH\zeta\left(\divH V
                         -\frac{\nablaH\zeta\cdot V_z}{J}\right)\right],
 \end{aligned}
\end{equation}
where 
$\mathcal B$ is given by
\begin{equation}\label{eq:transformed-diagnostic}
 \mathcal B(x_\rH,z,t)=-\int_0^z
   \left[\frac{H-\zeta}{2}\partial_zEL
       +\frac J2(L-EL)+\divH(QV)\right](x_\rH,s,t)\,\dd s.
\end{equation}
Here, the transformed vertical velocity $W$ is also defined diagnostically from \eqref{eq:physical-diagnostic}, that is,
\begin{equation*}
 \frac12\bigl([(H-z)w]\circ\phi\bigr)
                  =\mathcal B+R(EL+V\cdot\nablaH\zeta)
\end{equation*}
where $L$ is given by 
\begin{equation*}
 L(H,V)=-\frac2H\divH\int_0^1QV\,\dd z.
\end{equation*}
The transformed boundary conditions can likewise be written as
\begin{equation}\label{eq:transformed-boundary}
\begin{aligned}
 \mu\dz V|_{z=1}=f_3(H,V) \ \text{ and } \ 
 \dz V|_{z=0}=0
 \end{aligned}
\end{equation}
with
\begin{equation*}
 \begin{aligned}
f_3=\Bigl[&
 \mu\left(1-\frac{1+|\nablaH H|^2}{J}\right)V_z
 +\mu\bigl(\nablaH V+(\nablaH V)^\top\bigr)\nablaH H\\
 &+(\mu'-\mu)(\divH V)\nablaH H
 -\frac{\mu'}{J}(\nablaH H\cdot V_z)\nablaH H
 \Bigr]|_{z=1}.
 \end{aligned}
\end{equation*}
Expanding the horizontal viscous term in \eqref{eq:transformed-f2}
gives
\begin{equation*}
 \mu\divH\bigl((J-1)\nablaH V-V_z\otimes\nablaH\zeta\bigr)
 =\mu(J-1)\DeltaH V+\mu(\nablaH V)\nablaH J
 -\mu(\nablaH V_z)\nablaH\zeta-\mu V_z\DeltaH\zeta.
\end{equation*}
The last term  contains second
derivatives of the height, since
\begin{equation*}
 \DeltaH\zeta=z\DeltaH H_0+E\DeltaH(H-H_0).
\end{equation*}
However, the height equation
provides no smoothing. This creates a derivative
loss when estimating the coupled system. 
In contrast to Liu, Titi, and Xin~\cite[Remarks~1-2]{LTX},
who replace the full transformed viscosity by a linearized
operator supplemented with a correction restoring momentum
conservation, we retain the complete viscous contribution
and address the resulting derivative loss through a
Nash-Moser scheme based on tame estimates for the full
linearized problem.

\section{Nash-Moser iteration for the transformed problem}
\label{sec:nash-moser}

Our aim is to apply Nash-Moser directly to \eqref{eq:reduced-problem}
subject to \eqref{eq:transformed-boundary}. To do so, we first explain the main challenges of this approach on a toy model which discards all terms that do not cause immediate difficulties.

\subsection[The toy model]{The toy model corresponding to \eqref{eq:reduced-problem}
subject to \eqref{eq:transformed-boundary}}
\label{sec:toy-model}

We retain the height Hessian and the
height-gradient boundary term. For simplicity of the representation we also replace the Lam\'e type operator by the Laplacian. This gives the toy system
\begin{equation}\label{eq:toy-nonlinear}
 \left\{\begin{aligned}
 H_t+\divH\int_0^1(1-z)V\,\dd z&=0
       &&\text{in }(0,T)\times\T^2,\\
 (1-z)(V_t+\nablaH H)-\Delta V&=-zV_z\DeltaH H
       &&\text{in }(0,T)\times\Omega,\\
 H(0)=H_0\ \text{ and }\ V(0)&=V_0,
 \end{aligned}\right.
\end{equation}
with boundary conditions
\begin{equation*}
 V_z|_{z=1}=\bigl[(\nablaH V)\nablaH H\bigr]|_{z=1}
 \ \text{ and }\ V_z|_{z=0}=0.
\end{equation*}
At one correction step, fix the background height $H=1$ and a prescribed smooth
$V$ satisfying $V_z|_{z=0}=0$.
The unknowns $h$ and $\delta V$ are the height and horizontal-velocity
corrections to this background.
Differentiating \eqref{eq:toy-nonlinear}
gives their equations with prescribed sources $g_1,g_2,g_3$ as
\begin{equation*}
 \left\{\begin{aligned}
 h_t+\divH\int_0^1(1-z)\delta V\,\dd z&=g_1,\\
 (1-z)(\delta V_t+\nablaH h)-\Delta\delta V
       &=-zV_z\DeltaH h+g_2,\\
 h(0)=0\ \text{ and }\ \delta V(0)&=0,
 \end{aligned}\right.
\end{equation*}
with boundary conditions
\begin{equation*}
 \delta V_z|_{z=1}=(\nablaH V)|_{z=1}\nablaH h+g_3
 \ \text{ and }\ \delta V_z|_{z=0}=0.
\end{equation*}
To cancel the second derivatives of $h$, we use the good-unknown
mechanism of Alinhac \cite{Alinhac} and define $\widetilde V$ by
\begin{equation*}
\widetilde V=\delta V-zV_z h.
\end{equation*}
The system in $(h,\widetilde V)$ then reads as
\begin{equation}\label{eq:toy-first-good}
 \left\{\begin{aligned}
 h_t+\divH\int_0^1(1-z)\widetilde V\,\dd z+\divH(bh)&=g_1,\\
 (1-z)(\widetilde V_t+\nablaH h)-\Delta \widetilde V
   & =g_2+(1-z)\beta\left(\divH\int_0^1(1-z)\widetilde V\,\dd z+\divH(bh)-g_1\right)\\
   &\quad+2(\nablaH\beta)\nablaH h+
             \bigl(\Delta\beta-(1-z)\beta_t\bigr)h,\\
 h(0)=0\ \text{ and }\ \widetilde V(0)&=0,
 \end{aligned}\right.
\end{equation}
with
\begin{equation*}
 \widetilde V_z|_{z=1}=(\nablaH V)|_{z=1}\nablaH h-\beta_z|_{z=1}h+g_3
 \ \text{ and }\ \widetilde V_z|_{z=0}=0.
\end{equation*}
Here $\beta$ and $b$ are smooth given by 
\begin{equation*}
 \beta=zV_z\ \text{ and }\ b=\int_0^1(1-z)\beta\,\dd z.
\end{equation*}
The height Hessian has disappeared, but testing momentum with $\widetilde V$ produces a boundary contribution given by
$\langle(\nablaH V)|_{z=1}\nablaH h,\widetilde V|_{z=1}\rangle$.
Its trace estimate requires $h\in\rH^{1/2}(\T^2)$. In \cite{MR}, the corresponding fractional height estimate follows
from the kinematic condition and the dissipative trace control of
the full velocity. Here the problematic term is
$\divH\int_0^1(1-z)\widetilde V\,\dd z$ in the height equation,
since the available dissipation only gives
\begin{equation*}
 \|\divH\int_0^1(1-z)\widetilde V\,\dd z
       \|_{\rL^2(\T^2)}
 \leq C\|\widetilde V\|_{\rH^1(\Omega)^2},
\end{equation*}
whereas a direct fractional transport estimate would require
this term in $\rH^{1/2}(\T^2)$. We therefore choose the operator
$B$ below so that the momentum equation rewrites this term as
$-B\partial_t\widetilde V$ plus terms that can be estimated in
$\rH^{1/2}(\T^2)$. The modified height
$\widetilde h=h-B\widetilde V$ then absorbs this time derivative,
allowing its fractional estimate to close together with the
velocity energy estimate.
To obtain this control, we define the bounded operator $B$ by
\begin{equation*}
B \colon \rH^s(\Omega)^2 \to \rH^{s+1}(\T^2), \quad U \mapsto 
 BU=\divH\int_0^1(1-z)
              (1-\Delta_N)^{-1}\bigl((1-z)U\bigr)\,\dd z,
\end{equation*}
for $s \geq 0$.
Here $\Delta_N$ denotes the Neumann Laplacian. Elliptic regularity gives
\begin{equation*}
 \|BU\|_{\rH^{s+1}(\T^2)}\leq C_s\|U\|_{\rH^s(\Omega)^2}
 \ \text{ and }\
 \|BU\|_{\rH^1(\T^2)}\leq C\|\sqrt{1-z}\,U\|_{\rL^2(\Omega)^2}.
\end{equation*}
We regard horizontal functions as independent of $z$ when applying $B$.
Now introduce the corrected height $\widetilde h$ by
\begin{equation*}
\widetilde h=h-B\widetilde V.
\end{equation*}
The inverse in $B$ imposes homogeneous Neumann conditions, whereas $\widetilde V$
has a nonzero top derivative. We therefore define the boundary lifting
$\mathcal Nq$ by
\begin{equation*}
 (1-\Delta)\mathcal Nq=0\quad\text{in }\Omega,\quad
 (\mathcal Nq)_z|_{z=0}=0
 \ \text{ and }\ (\mathcal Nq)_z|_{z=1}=q.
\end{equation*}
Every smooth $Y$ with $Y_z|_{z=0}=0$ consequently satisfies
\begin{equation*}
 Y=(1-\Delta_N)^{-1}\bigl((1-\Delta)Y\bigr)
                   +\mathcal N(Y_z|_{z=1}).
\end{equation*}
Since $B$ is independent of time, we have
$\widetilde h_t=h_t-B\widetilde V_t$. Applying the elliptic decomposition to
momentum and using the height equation therefore gives
\begin{equation*}
 \begin{aligned}
 &\quad\widetilde h_t+\divH(b\widetilde h)
 \\&=g_1+B\nablaH h-\divH(bB\widetilde V)
 -\divH\int_0^1(1-z)(1-\Delta_N)^{-1}\biggl[
 g_2+2(\nablaH\beta)\nablaH h
       +\bigl(\Delta\beta-(1-z)\beta_t\bigr)h\\
 &\quad+\widetilde V+(1-z)\beta\left(
       \divH\int_0^1(1-s)\widetilde V(x_\rH,s,t)\,\dd s
       +\divH(bh)-g_1\right)\biggr]\,\dd z\\
 &\quad-\divH\int_0^1(1-z)\mathcal N
       \bigl[(\nablaH V)|_{z=1}\nablaH h
             -\beta_z|_{z=1}h+g_3\bigr]\,\dd z.
 \end{aligned}
\end{equation*}
Here $h=\widetilde h+B\widetilde V$ and $\widetilde h(0)=0$. 
We now aim to estimate $\widetilde h$ in the fractional space $\rH^{1/2}(\T^2)$. The horizontal resolvent estimate, followed by Cauchy-Schwarz in $z$, gives
\begin{equation*}
 \|\divH\int_0^1(1-z)(1-\Delta_N)^{-1}F\,\dd z
       \|_{\rH^{1/2}(\T^2)}
 \leq C\|F\|_{\rL^2((0,1),\rH^{-1/2}(\T^2)^2)}.
\end{equation*}
For the lifting, integration in $z$ gives
\begin{equation*}
 (1-\DeltaH)\int_0^1(1-z)\mathcal N q\,\dd z
       =(\mathcal N q)|_{z=1}-(\mathcal N q)|_{z=0},
\end{equation*}
so that
\begin{equation*}
 \|\divH\int_0^1(1-z)\mathcal N q\,\dd z
       \|_{\rH^{1/2}(\T^2)}
 \leq C\|\mathcal N q\|_{\rH^1(\Omega)}
 \leq C\|q\|_{\rH^{-1/2}(\T^2)}.
\end{equation*}
Combining this with boundedness of the operator $B$
yields
\begin{equation}\label{eq:H1/2}
 \begin{aligned}
 \|\widetilde h_t+\divH(b\widetilde h)\|_{\rH^{1/2}(\T^2)}
 \leq C\bigl(\|h\|_{\rH^{1/2}(\T^2)}
       +\|\widetilde V\|_{\rH^1(\Omega)^2}
       +\|g_1\|_{\rH^{1/2}(\T^2)}
       +\|g_2\|_{\rL^2(\Omega)^2}
       +\|g_3\|_{\rH^{-1/2}(\T^2)^2}\bigr),
 \end{aligned}
\end{equation}
where the constant $C$ depends on prescribed background.
We now test momentum in \eqref{eq:toy-first-good} with $\widetilde V$ and the
height equation with $h$. Integration by parts cancels their pressure
and column terms and gives the exact identity
\begin{equation*}
 \begin{aligned}
 &\quad \frac12\frac{\dd}{\dd t}\left(
       \|\sqrt{1-z}\,\widetilde V\|_2^2+\|h\|_2^2\right)
       +\|\nabla \widetilde V\|_2^2\\
 &=\int_{\T^2}g_1h\,\dd x_\rH
       -\frac12\int_{\T^2}(\divH b)h^2\,\dd x_\rH+\int_{\T^2}
       \left[(\nablaH V)|_{z=1}\nablaH h
             -\beta_z|_{z=1}h+g_3\right]
       \cdot \widetilde V|_{z=1}\,\dd x_\rH\\
 &\quad+\int_\Omega
       \left[(1-z)\beta\left(
       \divH\int_0^1(1-s)\widetilde V(x_\rH,s,t)\,\dd s
       +\divH(bh)-g_1\right)+g_2\right]\cdot \widetilde V\,\dd x\\
 &\quad+\int_\Omega
       \left[2(\nablaH\beta)\nablaH h
       +\bigl(\Delta\beta-(1-z)\beta_t\bigr)h\right]
       \cdot \widetilde V\,\dd x.
 \end{aligned}
\end{equation*}
We estimate the terms on the right-hand side separately using H\"older's inequality. First observe that
\begin{equation*}
 |\int_{\T^2}g_1h\,\dd x_\rH|
 +\frac12|\int_{\T^2}(\divH b)h^2\,\dd x_\rH|
 \leq C\bigl(\|g_1\|_{\rL^2(\T^2)}^2
                    +\|h\|_{\rL^2(\T^2)}^2\bigr).
\end{equation*}
For the boundary integral, bounded multiplication by the smooth
coefficients, Sobolev duality and the trace theorem yield
\begin{equation*}
 \begin{aligned}
 |\int_{\T^2}
       \left[(\nablaH V)|_{z=1}\nablaH h
             -\beta_z|_{z=1}h+g_3\right]
       \cdot \widetilde V|_{z=1}\,\dd x_\rH|
\leq C\bigl(\|h\|_{\rH^{1/2}(\T^2)}
             +\|g_3\|_{\rH^{-1/2}(\T^2)^2}\bigr)
                         \|\widetilde V\|_{\rH^1(\Omega)^2}.
 \end{aligned}
\end{equation*}
Next, we estimate
\begin{equation*}
 |\int_\Omega(1-z)(\beta\cdot \widetilde V)
       \divH\int_0^1(1-s)\widetilde V(x_\rH,s,t)\,\dd s\,\dd x|
 \leq C\|\nabla \widetilde V\|_{\rL^2(\Omega)}
              \|\sqrt{1-z}\,\widetilde V\|_{\rL^2(\Omega)^2}.
\end{equation*}
For the remaining volume terms, we move the horizontal derivatives
of $h$ onto the smooth coefficients and $\widetilde V$ by integration by parts.
We consequently obtain
\begin{equation*}
 \begin{aligned}
 &\quad|\int_\Omega
       \left[(1-z)\beta\bigl(\divH(bh)-g_1\bigr)+g_2\right]
                        \cdot \widetilde V\,\dd x|
 +|\int_\Omega
       \left[2(\nablaH\beta)\nablaH h
       +\bigl(\Delta\beta-(1-z)\beta_t\bigr)h\right]
                        \cdot \widetilde V\,\dd x|\\
 &\leq C\bigl(\|h\|_{\rL^2(\T^2)}
             +\|g_1\|_{\rL^2(\T^2)}
             +\|g_2\|_{\rL^2(\Omega)^2}\bigr)
                        \|\widetilde V\|_{\rH^1(\Omega)^2}.
 \end{aligned}
\end{equation*}
To absorb the velocity factors, we use the weighted Poincar\'e
inequality
\begin{equation*}
 \|\widetilde V\|_{\rH^1(\Omega)^2}^2
 \leq C\bigl(\|\nabla \widetilde V\|_{\rL^2(\Omega)}^2
             +\|\sqrt{1-z}\,\widetilde V\|_{\rL^2(\Omega)^2}^2\bigr).
\end{equation*}
Combining the preceding bounds with Young's inequality and
$\rH^{1/2}(\T^2)\hookrightarrow\rL^2(\T^2)$ therefore gives
\begin{equation*}
 \begin{aligned}
 &\quad\frac12\frac{\dd}{\dd t}\left(
       \|\sqrt{1-z}\,\widetilde V\|_{\rL^2(\Omega)^2}^2
       +\|h\|_{\rL^2(\T^2)}^2\right)
       +\frac12\|\nabla \widetilde V\|_{\rL^2(\Omega)}^2\\
 &\leq C\bigl(
       \|\sqrt{1-z}\,\widetilde V\|_{\rL^2(\Omega)^2}^2
       +\|h\|_{\rH^{1/2}(\T^2)}^2
       +\|g_1\|_{\rL^2(\T^2)}^2
       +\|g_2\|_{\rL^2(\Omega)^2}^2
       +\|g_3\|_{\rH^{-1/2}(\T^2)^2}^2\bigr).
 \end{aligned}
\end{equation*}
Here $C$ depends on the prescribed background.
We now combine the preceding estimate with the fractional energy
estimate for $\widetilde h$. For this purpose, we define $E$ by
\begin{equation*}
 E=\frac12\|\sqrt{1-z}\,\widetilde V\|_{\rL^2(\Omega)}^2
       +\frac12\|h\|_{\rL^2(\T^2)}^2
       +\frac12\|\widetilde h\|_{\rH^{1/2}(\T^2)}^2.
\end{equation*}
Since $h=\widetilde h+B\widetilde V$, it holds that
\begin{equation*}
 \|h\|_{\rH^{1/2}(\T^2)}^2
 \leq 2\|\widetilde h\|_{\rH^{1/2}(\T^2)}^2
       +C\|B\widetilde V\|_{\rH^1(\T^2)}^2
 \leq C\bigl(\|\widetilde h\|_{\rH^{1/2}(\T^2)}^2
       +\|\sqrt{1-z}\,\widetilde V\|_{\rL^2(\Omega)}^2\bigr)
 \leq CE.
\end{equation*}
To differentiate the fractional energy, we express the fractional norm via the Bessel potential norm, using $(1-\DeltaH)^{1/4}$ so that
\begin{equation*}
 \begin{aligned}
 &\quad\frac12\frac{\dd}{\dd t}\|\widetilde h\|_{\rH^{1/2}(\T^2)}^2
 \\&=\int_{\T^2}(1-\DeltaH)^{1/4}\widetilde h\,
       (1-\DeltaH)^{1/4}
       \bigl(\widetilde h_t+\divH(b\widetilde h)\bigr)\,\dd x_\rH-\int_{\T^2}(1-\DeltaH)^{1/4}\widetilde h\,
       (1-\DeltaH)^{1/4}\divH(b\widetilde h)\,\dd x_\rH.
 \end{aligned}
\end{equation*}
For the first integral, H\"older's inequality directly gives
\begin{equation*}
 \begin{aligned}
 |\int_{\T^2}(1-\DeltaH)^{1/4}\widetilde h\,
       (1-\DeltaH)^{1/4}
       \bigl(\widetilde h_t+\divH(b\widetilde h)\bigr)\,\dd x_\rH|
 \leq\|\widetilde h\|_{\rH^{1/2}(\T^2)}
       \|\widetilde h_t+\divH(b\widetilde h)\|_{\rH^{1/2}(\T^2)}.
 \end{aligned}
\end{equation*}
For the second integral, we first expand
$\divH(b\widetilde h)=b\cdot\nablaH\widetilde h+(\divH b)\widetilde h$.
To separate the transport term, we define the commutator by
\begin{equation*}
 \bigl[(1-\DeltaH)^{1/4},b\cdot\nablaH\bigr]\widetilde h
 =(1-\DeltaH)^{1/4}(b\cdot\nablaH\widetilde h)
       -b\cdot\nablaH(1-\DeltaH)^{1/4}\widetilde h.
\end{equation*}
Since $(1-\DeltaH)^{1/4}$ commutes with horizontal derivatives,
integration by parts gives the exact identity
\begin{equation*}
 \begin{aligned}
 \int_{\T^2}(1-\DeltaH)^{1/4}\widetilde h\,
       (1-\DeltaH)^{1/4}\divH(b\widetilde h)\,\dd x_\rH
 &=-\frac12\int_{\T^2}(\divH b)
       |(1-\DeltaH)^{1/4}\widetilde h|^2\,\dd x_\rH\\
 &\quad+\int_{\T^2}(1-\DeltaH)^{1/4}\widetilde h\,
       \bigl[(1-\DeltaH)^{1/4},b\cdot\nablaH\bigr]\widetilde h
       \,\dd x_\rH\\
 &\quad+\int_{\T^2}(1-\DeltaH)^{1/4}\widetilde h\,
       (1-\DeltaH)^{1/4}\bigl((\divH b)\widetilde h\bigr)\,\dd x_\rH.
 \end{aligned}
\end{equation*}
Since $b$ is smooth, the commutator
$\bigl[(1-\DeltaH)^{1/4},b\cdot\nablaH\bigr]$
is a pseudodifferential operator of order $1/2$. It then holds that
\begin{equation*}
 \|\bigl[(1-\DeltaH)^{1/4},b\cdot\nablaH\bigr]\widetilde h
       \|_{\rL^2(\T^2)}
 \leq C_b\|\widetilde h\|_{\rH^{1/2}(\T^2)},
\end{equation*}
where $C_b$ depends on finitely many smooth norms of $b$.
We refer to \cite[Chapter~I, Sections~2, 3,~8]{TaylorShortCourse}
for the commutator calculus, the Sobolev mapping properties and
their extension to compact manifolds.
For the remaining multiplication term, H\"older's inequality
and the product rule give
\begin{equation*}
 \begin{aligned}
 \|(\divH b)\widetilde h\|_{\rL^2(\T^2)}
 \leq \|\divH b\|_{\rL^\infty(\T^2)}
                         \|\widetilde h\|_{\rL^2(\T^2)} \ \text{ and } \
 \|(\divH b)\widetilde h\|_{\rH^1(\T^2)}
 \leq C\|\divH b\|_{\rW^{1,\infty}(\T^2)}
                         \|\widetilde h\|_{\rH^1(\T^2)}
 \end{aligned}
\end{equation*}
so that by interpolation, one has
\begin{equation*}
 \begin{aligned}
 \|(1-\DeltaH)^{1/4}
       \bigl((\divH b)\widetilde h\bigr)\|_{\rL^2(\T^2)}
 &=\|(\divH b)\widetilde h\|_{\rH^{1/2}(\T^2)}\\
 &\leq C\|\divH b\|_{\rW^{1,\infty}(\T^2)}
                         \|\widetilde h\|_{\rH^{1/2}(\T^2)}
 \leq C_b\|\widetilde h\|_{\rH^{1/2}(\T^2)}.
 \end{aligned}
\end{equation*}
These bounds control all three terms in the preceding identity
and give
\begin{equation*}
 |\int_{\T^2}(1-\DeltaH)^{1/4}\widetilde h\,
       (1-\DeltaH)^{1/4}\divH(b\widetilde h)\,\dd x_\rH|
 \leq C_b
                    \|\widetilde h\|_{\rH^{1/2}(\T^2)}^2.
\end{equation*}
Combining the estimates for the two integrals therefore yields
\begin{equation*}
 \frac12\frac{\dd}{\dd t}\|\widetilde h\|_{\rH^{1/2}(\T^2)}^2
 \leq C_b
                    \|\widetilde h\|_{\rH^{1/2}(\T^2)}^2
 +\|\widetilde h\|_{\rH^{1/2}(\T^2)}
       \|\widetilde h_t+\divH(b\widetilde h)\|_{\rH^{1/2}(\T^2)}
\end{equation*}
and from \eqref{eq:H1/2} it follows that
\begin{equation*}
 \begin{aligned}
 &\quad\|\widetilde h\|_{\rH^{1/2}(\T^2)}
       \|\widetilde h_t+\divH(b\widetilde h)\|_{\rH^{1/2}(\T^2)}
\\& \leq C\big (E+\sqrt E\,\|\widetilde V\|_{\rH^1(\Omega)}
       +\sqrt E\bigl(
          \|g_1\|_{\rH^{1/2}(\T^2)}
         +\|g_2\|_{\rL^2(\Omega)}
         +\|g_3\|_{\rH^{-1/2}(\T^2)}\bigr)\big ).
 \end{aligned}
\end{equation*}
The weighted Poincar\'e inequality gives
$\|\widetilde V\|_{\rH^1(\Omega)}
 \leq C(\|\nabla \widetilde V\|_{\rL^2(\Omega)}+\sqrt E)$.
Young's inequality therefore yields
\begin{equation*}
 \begin{aligned}
 \frac12\frac{\dd}{\dd t}\|\widetilde h\|_{\rH^{1/2}(\T^2)}^2
 \leq\frac14\|\nabla \widetilde V\|_{\rL^2(\Omega)}^2+C\big (E
 +\bigl(\|g_1\|_{\rH^{1/2}(\T^2)}^2
         +\|g_2\|_{\rL^2(\Omega)}^2
         +\|g_3\|_{\rH^{-1/2}(\T^2)}^2\bigr)\big ).
 \end{aligned}
\end{equation*}
Adding this inequality to the preceding estimate for $(h,\widetilde V)$,
and using $\|h\|_{\rH^{1/2}(\T^2)}^2\leq CE$, gives
\begin{equation*}
 E'(t)+\frac14\|\nabla \widetilde V\|_{\rL^2(\Omega)}^2
 \leq CE(t)+C\bigl(
          \|g_1\|_{\rH^{1/2}(\T^2)}^2
         +\|g_2\|_{\rL^2(\Omega)}^2
         +\|g_3\|_{\rH^{-1/2}(\T^2)}^2\bigr).
\end{equation*}
Here the constants depend on finitely many smooth norms of the prescribed
background $V$. This is the starting point of the tame estimate for the Nash-Moser scheme. From here it has to be upgraded to arbitrary Sobolev regularity. The key steps described above now have to be translated to the full transformed problem  \eqref{eq:reduced-problem}
subject to \eqref{eq:transformed-boundary} which is essentially carried out in the remainder of this article.

\subsection[The linear correction problem]{The Nash-Moser scheme for \eqref{eq:reduced-problem}
subject to \eqref{eq:transformed-boundary}}
We work on a fixed interval $[0,T]$. Here we denote by $n$ the iteration step and write $H=H_n$ and $V=V_n$. Moreover, we denote
the height and velocity corrections by $h$ and $\delta V$, so that
\begin{equation*}
    H_{n+1} = H + h \ \text{ and } \ V_{n+1} = V + \delta V.
\end{equation*}
For simplicity of the notation, we define the vertical average $\mathcal Q$, acting componentwise on
smooth scalar, vector and matrix fields $Y$, by
\begin{equation}\label{eq:good-column}
 \mathcal QY=\int_0^1Y\,\dd z.
\end{equation}
For the viscous terms in the interior and at the boundary, we
define $\mathcal S_HY$ and $\mathcal T_HY$ by
\begin{equation}\label{eq:good-stress-flux}
 \begin{aligned}
 \mathcal S_HY&=
 \mu\left(\nablaH Y+(\nablaH Y)^\top
       -\frac{Y_z\otimes\nablaH\zeta
                     +\nablaH\zeta\otimes Y_z}{J}\right)
 +(\mu'-\mu)\left(\divH Y
       -\frac{\nablaH\zeta\cdot Y_z}{J}\right)I_2 \ \text{ and }\\
 \mathcal T_HY&=\frac\mu JY_z-(\mathcal S_HY)\nablaH\zeta.
 \end{aligned}
\end{equation}
In this notation, the full transformed viscosity is
$\divH(J\mathcal S_HY)+\partial_z\mathcal T_HY$, and the
top boundary remainder satisfies
\begin{equation*}
 f_3(H,Y)=\left[\mu Y_z-\mathcal T_HY\right]_{z=1}.
\end{equation*}

We define $\delta f_i=Df_i(H,V)[h,\delta V]$ for $i=1,2,3$.
Using \eqref{eq:good-column}, linearizing
\eqref{eq:reduced-problem} gives
\begin{equation}\label{eq:raw-linearized-system}
 \left\{\begin{aligned}
 h_t+\frac2{H_0}\divH\mathcal Q(Q_0\delta V)
        &=g_1+\delta f_1 &&\text{in }(0,T)\times\T^2,\\
 Q_0(\delta V_t+\nablaH h)-\mu\Delta\delta V
                         -\mu'\nablaH\divH\delta V
        &=g_2+\delta f_2 &&\text{in }(0,T)\times\Omega,\\
 h(0)=0\ \text{ and }\ \delta V(0)&=0.
 \end{aligned}\right.
\end{equation}
Its boundary conditions are
\begin{equation}\label{eq:raw-linearized-boundary}
 \mu\dz\delta V|_{z=1}=g_3+\delta f_3
 \ \text{ and }\
 \dz\delta V|_{z=0}=g_4.
\end{equation}
To calculate the right-hand sides, we use
\begin{equation}\label{eq:raw-coefficient-variations}
 \delta\zeta=Eh,\quad\delta J=\partial_zEh
 \ \text{ and }\
 \delta Q=\frac{H-\zeta}{2}\partial_zEh+\frac J2(h-Eh).
\end{equation}
Differentiating \eqref{eq:transformed-f1} and using
$L=-2H^{-1}\divH\mathcal Q(QV)$ yields
\begin{equation*}
 \delta f_1=\frac2{H_0}\divH\mathcal Q(Q_0\delta V)
       -\frac2H\divH\mathcal Q(\delta Q\,V+Q\delta V)
       -\frac LHh
\end{equation*}
and using the notation in \eqref{eq:good-stress-flux}, differentiation
of \eqref{eq:transformed-f2} gives
\begin{equation}\label{eq:raw-df2}
 \begin{aligned}
 \delta f_2={}&(Q_0-Q)(\delta V_t+\nablaH h)
       -\delta Q(V_t+V\cdot\nablaH V+\nablaH H)
       -Q\bigl(\delta V\cdot\nablaH V+V\cdot\nablaH\delta V\bigr)\\
 &-\delta\mathcal B\,V_z-\mathcal B\,\delta V_z
       +\divH(J\mathcal S_H\delta V)+\partial_z\mathcal T_H\delta V
       -\mu\Delta\delta V-\mu'\nablaH\divH\delta V\\
 &+\divH\bigl(\delta J\,\mathcal S_HV
                   +J(\delta_H\mathcal S_H[h])V\bigr)
       +\partial_z\bigl((\delta_H\mathcal T_H[h])V\bigr).
 \end{aligned}
\end{equation}
Similarly, differentiating the top boundary remainder gives
\begin{equation*}
 \delta f_3=\left[
       \mu\delta V_z-\mathcal T_H\delta V
                      -(\delta_H\mathcal T_H[h])V
       \right]|_{z=1}.
\end{equation*}
Here $\delta_H\mathcal S_H[h]$ and
$\delta_H\mathcal T_H[h]$ denote the variations of the
coefficients with the argument $Y$ held fixed. By
\eqref{eq:raw-coefficient-variations}, they are explicitly given by
\begin{equation*}
 \begin{aligned}
 (\delta_H\mathcal S_H[h])Y&=
 -\frac\mu J\bigl(Y_z\otimes\nablaH Eh
                            +\nablaH Eh\otimes Y_z\bigr)
 +\frac{\mu\delta J}{J^2}
       \bigl(Y_z\otimes\nablaH\zeta+\nablaH\zeta\otimes Y_z\bigr)\\
 &\quad+(\mu'-\mu)\left(-\frac{\nablaH Eh\cdot Y_z}{J}
             +\frac{\delta J}{J^2}\nablaH\zeta\cdot Y_z\right)I_2 \ \text{ and }\\
 (\delta_H\mathcal T_H[h])Y
 &=-\frac{\mu\delta J}{J^2}Y_z
       -\bigl((\delta_H\mathcal S_H[h])Y\bigr)\nablaH\zeta
       -(\mathcal S_HY)\nablaH Eh.
 \end{aligned}
\end{equation*}
To make sense of every term in \eqref{eq:raw-df2}, we also differentiate $L$ to obtain
\begin{equation*}
 \delta L=-\frac2H\divH\mathcal Q(\delta Q\,V+Q\delta V)
                         -\frac LHh.
\end{equation*}
so that by \eqref{eq:transformed-diagnostic} we find 
\begin{equation*}
 \begin{aligned}
 \delta\mathcal B=-\int_0^z\biggl[&
       \frac{h-Eh}{2}\partial_zEL
       +\frac{L-EL}{2}\partial_zEh
       +\frac{H-\zeta}{2}\partial_zE(\delta L)\\
       &+\frac J2\bigl(\delta L-E(\delta L)\bigr)
       +\divH(\delta Q\,V+Q\delta V)
       \biggr](x_\rH,s,t)\,\dd s.
 \end{aligned}
\end{equation*}
The pressure and flux identities at the boundary also give
$\delta P|_{z=1}=0$ and
$\delta\mathcal B|_{z=0}=\delta\mathcal B|_{z=1}=0$.
To express the correction problem compactly, we define the
nonlinear residual $\mathcal F = (\mathcal{F}_H, \mathcal F_V, \mathcal F_b)^\top$ by
\begin{equation}\label{eq:full-residual}
 \mathcal F(H,V)=
 \begin{pmatrix}
 H_t+\dfrac2{H_0}\divH\mathcal Q(Q_0V)-f_1(H,V)\\[1mm]
 Q_0(V_t+\nablaH H)-\mu\Delta V
                         -\mu'\nablaH\divH V-f_2(H,V)\\[1mm]
 \mu\dz V|_{z=1}-f_3(H,V)\\[1mm]
 \dz V|_{z=0}
 \end{pmatrix}.
\end{equation}
Then it holds that
\begin{equation*}
 \begin{aligned}
 \mathcal F_H(H,V)=H_t-L\ \text{ and }\
 \mathcal F_V(H,V)
 =Q(V_t+V\cdot\nablaH V+\nablaH H)+\mathcal B V_z
       -\divH(J\mathcal S_HV)-\partial_z\mathcal T_HV.
 \end{aligned}
\end{equation*}
For an unsmoothed Newton
step, the sources are given by
\begin{equation*}
(g_{1,n},g_{2,n},g_{3,n},g_{4,n})^\top
 =-\mathcal F(H_n,V_n).
\end{equation*}
The corresponding
Nash-Moser iteration of\eqref{eq:raw-linearized-system} subject to
\eqref{eq:raw-linearized-boundary} takes the form
\begin{equation*}
 D\mathcal F(H_n,V_n)[h,\delta V]
 =(g_{1,n},g_{2,n},g_{3,n},g_{4,n})^\top,
 \quad h(0)=0\ \text{ and }\ \delta V(0)=0.
\end{equation*}
The coefficient variations in \eqref{eq:raw-df2} still produce
second derivatives of $Eh$, including
\begin{equation*}
 -\mu\DeltaH(Eh)V_z-\mu'\nablaH^2(Eh)V_z.
\end{equation*}
As in the toy model, we cancel those terms using Alinhac's good
unknown \cite{Alinhac}, which we define by
\begin{equation}\label{eq:good-unknown}
 \widetilde V=\delta V-\frac{Eh}{J}V_z.
\end{equation}
Substituting \eqref{eq:good-unknown} into
\eqref{eq:raw-linearized-system} subject to
\eqref{eq:raw-linearized-boundary}, and using the height equation
to eliminate $h_t$ from the momentum remainder, gives
\begin{equation}\label{eq:off-exact-momentum}
 \left\{\begin{aligned}
 h_t+\frac2{H_0}\divH\mathcal Q(Q_0\widetilde V)
 &=\widetilde g_1+\widetilde f_1
       &&\text{in }(0,T)\times\T^2,\\
 Q_0(\widetilde V_t+\nablaH h)-\mu\Delta\widetilde V
                         -\mu'\nablaH\divH\widetilde V
 &=\widetilde g_2+\widetilde f_2
       &&\text{in }(0,T)\times\Omega,\\
 h(0)=0\ \text{ and }\ \widetilde V(0)&=0.
 \end{aligned}\right.
\end{equation}
Its boundary conditions are
\begin{equation}\label{eq:good-boundary}
 \mu\dz\widetilde V|_{z=1}=\widetilde g_3+\widetilde f_3
 \ \text{ and }\
 \dz\widetilde V|_{z=0}=\widetilde g_4+\widetilde f_4.
\end{equation}
The prescribed sources are now given by
\begin{equation*}
 \widetilde g_1=g_1,\quad
 \widetilde g_2=g_2-\frac QJ(Eg_1)V_z,\quad
 \widetilde g_3=g_3
 \ \text{ and }\
 \widetilde g_4=g_4.
\end{equation*}
Moreover, the terms $\widetilde f_i$ are given by
\begin{equation*}
 \widetilde f_1
 =\frac2{H_0}\divH\mathcal Q(Q_0\widetilde V)
  -\frac2H\divH\mathcal Q\left(Q\widetilde V+\frac{Jh}{2}V\right)
  -\frac LHh.
\end{equation*}
and
\begin{equation}\label{eq:good-g2}
 \begin{aligned}
 \widetilde f_2={}&
 (Q_0-Q)(\widetilde V_t+\nablaH h)
 +\divH(J\mathcal S_H\widetilde V)
 +\partial_z\mathcal T_H\widetilde V
 -\mu\Delta\widetilde V-\mu'\nablaH\divH\widetilde V\\
 &-Q\bigl(V\cdot\nablaH\widetilde V
                  +\widetilde V\cdot\nablaH V\bigr)
       -\mathcal B\widetilde V_z
       -\frac{Jh}{2}\bigl(V_t+V\cdot\nablaH V+\nablaH H\bigr)\\
 &-\biggl[
       \frac\zeta H\divH\mathcal Q\left(Q\widetilde V+\frac{Jh}{2}V\right)
       -\divH\int_0^z\left(Q\widetilde V+\frac{Jh}{2}V\right)
                           (x_\rH,s,t)\,\dd s
       +\frac h2\left(\frac{\zeta L}{H}-EL\right)
       \biggr]V_z\\
 &+\frac{QEh}{J^2}\partial_z\bigl(E\mathcal F_H(H,V)\bigr)V_z
       -\partial_z\left[\frac{Eh}{J}\mathcal F_V(H,V)\right],
 \end{aligned}
\end{equation}
as well as 
\begin{equation*}
 \begin{aligned}
 \widetilde f_3=\left[
       \mu\widetilde V_z-\mathcal T_H\widetilde V
       +J(\mathcal S_HV)\nablaH\left(\frac{Eh}{J}\right)
       -\frac{Eh}{J}\partial_z\mathcal T_HV
       \right]|_{z=1}\ \text{ and }\
 \widetilde f_4=-\left[\frac{\partial_zEh}{J}V_z\right]|_{z=0}.
 \end{aligned}
\end{equation*} 
Note that in
\eqref{eq:good-g2}, the residuals $\mathcal F_H(H,V)$ and
$\mathcal F_V(H,V)$ are prescribed background functions.
The cancellation of the second derivatives of $Eh$ follows
from
\begin{equation*}
 \begin{aligned}
 \delta\bigl(\divH(J\mathcal S_HV)+\partial_z\mathcal T_HV\bigr)
 =\divH(J\mathcal S_H\widetilde V)
       +\partial_z\mathcal T_H\widetilde V
       +\partial_z\left[
          \frac{Eh}{J}
          \bigl(\divH(J\mathcal S_HV)+\partial_z\mathcal T_HV\bigr)
          \right].
 \end{aligned}
\end{equation*}
To express the new correction problem in the same form, we
define the linear operator $\widetilde{\mathcal F}(H,V)$ by
\begin{equation*}
 \widetilde{\mathcal F}(H,V)[h,\widetilde V]
 =\begin{pmatrix}
 h_t+\dfrac2{H_0}\divH\mathcal Q(Q_0\widetilde V)-\widetilde f_1\\[1mm]
 Q_0(\widetilde V_t+\nablaH h)-\mu\Delta\widetilde V
                  -\mu'\nablaH\divH\widetilde V-\widetilde f_2\\[1mm]
 \mu\dz\widetilde V|_{z=1}-\widetilde f_3\\[1mm]
 \dz\widetilde V|_{z=0}-\widetilde f_4
 \end{pmatrix}.
\end{equation*}
Since $J>0$, the substitution
$\delta V=\widetilde V+(Eh)V_z/J$ is reversible. Together
with the stated change of sources, it gives an equivalent
linear correction problem. At the $n$-th step, this problem
therefore takes the form
\begin{equation*}
 \widetilde{\mathcal F}(H_n,V_n)[h,\widetilde V]
 =\begin{pmatrix}
 \widetilde g_{1,n}\\\widetilde g_{2,n}\\
 \widetilde g_{3,n}\\\widetilde g_{4,n}
 \end{pmatrix}
 =\begin{pmatrix}
 g_{1,n}\\
 g_{2,n}-\dfrac{Q_n}{J_n}(Eg_{1,n})\partial_zV_n\\
 g_{3,n}\\g_{4,n}
 \end{pmatrix},
 \quad h(0)=0\ \text{ and }\ \widetilde V(0)=0.
\end{equation*}
In analogy with the toy model, we introduce a modified height in order to establish the basic tame energy estimate. For this
purpose, we first define the shifted viscous operator
$\mathcal L_H$ at a prescribed smooth background height $H$
by 
\begin{equation*}
    \begin{aligned}
   \mathcal L_H \colon \mathrm D(\mathcal L_H)    & \to \rH^s(\Omega)^2 \ \text{ with } \  \mathrm D(\mathcal L_H)=
 \left\{U\in\rH^{s+2}(\Omega)^2 \colon
       \mathcal T_HU|_{z=1}=0,\ U_z|_{z=0}=0\right\} \\
 U& \mapsto\mathcal L_HU
       =-\divH(J\mathcal S_HU)-\partial_z\mathcal T_HU+QU.
    \end{aligned}
\end{equation*}
for $s \geq 0$.
The coefficient $Q$ is positive in the interior and vanishes at the top.
To justify the inverse, we define its form by
\begin{equation}\label{eq:form}
 a_H(U,Y)=\int_\Omega\left[
 J\operatorname{tr}\bigl((\mathcal S_HU)^\top\nablaH Y\bigr)
       +\mathcal T_HU\cdot Y_z+QU\cdot Y\right]\,\dd x.
\end{equation}
The physical strain form, Korn's inequality and the positive interior
weight give
\begin{equation*}
 |a_H(U,Y)|\leq C\|U\|_{\rH^1(\Omega)^2}\|Y\|_{\rH^1(\Omega)^2}
 \ \text{ and }\
 a_H(U,U)\geq c\|U\|_{\rH^1(\Omega)^2}^2.
\end{equation*}
Smooth conormal
elliptic regularity, see \cite{ADN}, then gives
\begin{equation*}
 \mathcal L_H^{-1}\in
 \mathcal L\bigl(\rH^s(\Omega)^2,\rH^{s+2}(\Omega)^2\bigr),
 \quad s\geq0.
\end{equation*}
With this in mind, we define the analogue of the operator $B$ in the toy problem $B_H$ by 
\begin{equation*}
B_H \colon \rH^{s}(\Omega)^2 \to \rH^{s+1}(\T^2), \quad 
 U \mapsto B_HU
       =\frac2H\divH\mathcal Q\bigl(Q\mathcal L_H^{-1}(QU)\bigr).
\end{equation*}
Vertical integration is bounded at every nonnegative Sobolev order,
and horizontal divergence lowers the order by one. Thus
\begin{equation*}
 \|B_HU\|_{\rH^{s+1}(\T^2)}
 \leq C_s\|\mathcal L_H^{-1}(QU)\|_{\rH^{s+2}(\Omega)^2}
 \leq C_s\|U\|_{\rH^s(\Omega)^2}.
\end{equation*}
At $s=0$, the same estimate and $\|QU\|_2\leq C\|\sqrt Q\,U\|_2$
give the following bounds.

\begin{cor}\label{lem:height-correction-operator}
Suppose that the prescribed background $H$ is smooth and that $T>0$ is chosen small enough such that 
$J$ and $Q/(1-z)$ are bounded away from zero. Then, for every
$s\geq0$, the operator $B_H$ is bounded from
$\rH^s(\Omega)^2$ to $\rH^{s+1}(\T^2)$ and satisfies
\begin{equation*}
 \|B_HU\|_{\rH^1(\T^2)}
       \leq C\|\sqrt Q\,U\|_{\rL^2(\Omega)^2}.
\end{equation*}
The constants depend on finitely many background norms and
the positive chart bounds.
\end{cor}
We now define the modified height $\widetilde h$ by
\begin{equation}\label{eq:modified-height}
 \widetilde h=h-B_H\widetilde V.
\end{equation}
To compactly write its equation we define several helpful quantities. First set $b_1$ and $b_2$ to be 
\begin{equation*}
 b_1=\frac1H\mathcal Q(JV)
 \ \text{ and }\
  b_2 =\frac1H\divH\mathcal Q(JV)+\frac LH,
\end{equation*}
so that the height equation in \eqref{eq:off-exact-momentum} becomes
\begin{equation*}
 h_t+b_1\cdot\nablaH h+ b_2 h
       +\frac2H\divH\mathcal Q(Q\widetilde V)=g_1.
\end{equation*}
Next, define
$\mathcal G(h,U)$ by
\begin{equation*}
 \begin{aligned}
 \mathcal G(h,U)={}&\widetilde g_2+QU
       -Q\bigl(V\cdot\nablaH U+U\cdot\nablaH V\bigr)
       -\mathcal B U_z
       -\frac{Jh}{2}\bigl(V_t+V\cdot\nablaH V+\nablaH H\bigr)\\
 &-\biggl[
       \frac\zeta H\divH\mathcal Q\left(QU+\frac{Jh}{2}V\right)
       -\divH\int_0^z\left(QU+\frac{Jh}{2}V\right)
                           (x_\rH,s,t)\,\dd s
       +\frac h2\left(\frac{\zeta L}{H}-EL\right)
       \biggr]V_z\\
 &+\frac{QEh}{J^2}\partial_z\bigl(E\mathcal F_H(H,V)\bigr)V_z
       -\partial_z\left[\frac{Eh}{J}\mathcal F_V(H,V)\right].
 \end{aligned}
\end{equation*}
Thus the momentum equation reads
\begin{equation*}
 Q(\widetilde V_t+\nablaH h)+\mathcal L_H\widetilde V
       =\mathcal G(h,\widetilde V).
\end{equation*}
Its boundary values are given by 
\begin{equation*}
 \mathcal T_H\widetilde V|_{z=1}=G_1(h)
 \ \text{ and }\
 \widetilde V_z|_{z=0}=G_0(h),
\end{equation*}
where
\begin{equation*}
 \begin{aligned}
 G_1(h)=g_3+\left[
       J(\mathcal S_HV)\nablaH\left(\frac{Eh}{J}\right)
       -\frac{Eh}{J}\partial_z\mathcal T_HV\right]|_{z=1}
 \ \text{ and }\
 G_0(h)=g_4-\left[\frac{\partial_zEh}{J}V_z\right]|_{z=0}.
 \end{aligned}
\end{equation*}
Since $\mathcal L_H^{-1}$ uses homogeneous boundary
conditions, we retain these values by defining the boundary
lifting $\mathcal N_H(G_1,G_0)$ as the unique solution $Y$ of
\begin{equation*}
 \mathcal L_HY=0\quad\text{in }\Omega, \quad  \mathcal T_HY|_{z=1}=G_1
 \ \text{ and }\
 Y_z|_{z=0}=G_0.
\end{equation*}
Then it holds that
\begin{equation*}
 \widetilde V
 =\mathcal L_H^{-1}\bigl(
       \mathcal G(h,\widetilde V)-Q\nablaH h-Q\widetilde V_t\bigr)
       +\mathcal N_H(G_1(h),G_0(h)).
\end{equation*}
Differentiating \eqref{eq:modified-height} gives
$\widetilde h_t=h_t-B_H\widetilde V_t-(\partial_tB_H)\widetilde V$ and we find 
\begin{equation}\label{eq:modified-height-equation}
 \begin{aligned}
 \widetilde h_t+b_1\cdot\nablaH\widetilde h+ b_2 \widetilde h
 ={}&g_1+B_H\nablaH h
       -(\partial_tB_H)\widetilde V
       -b_1\cdot\nablaH(B_H\widetilde V)- b_2 B_H\widetilde V\\
 &-\frac2H\divH\mathcal Q\bigl(
       Q\mathcal L_H^{-1}\mathcal G(h,\widetilde V)\bigr)
       -\frac2H\divH\mathcal Q\bigl(
       Q\mathcal N_H(G_1(h),G_0(h))\bigr).
 \end{aligned}
\end{equation}
We now estimate the right-hand side of
the above equation in
$\rH^{1/2}(\T^2)$. 
In the following negative-order volume estimates, we use the convention
\begin{equation*}
 \rH^{-1/2}(\Omega)^2=(\rH^{1/2}(\Omega)^2)^*.
\end{equation*}
Thus the functional is specified on all tests, including those with
nonzero boundary values. The weak formulation of $\mathcal L_H$, coercivity and
standard elliptic regularity give
\begin{equation*}
 \mathcal L_H^{-1}\in
 \mathcal L\bigl((\rH^1(\Omega)^2)^*,\rH^1(\Omega)^2\bigr)
 \ \text{ and }\
 \mathcal L_H^{-1}\in
 \mathcal L\bigl(\rL^2(\Omega)^2,\rH^2(\Omega)^2\bigr).
\end{equation*}
Interpolating these bounds yields
\begin{equation*}
 \|\mathcal L_H^{-1}F\|_{\rH^{3/2}(\Omega)^2}
 \leq C\|F\|_{\rH^{-1/2}(\Omega)^2}
\end{equation*}
so that 
\begin{equation}\label{eq:basic-column-inverse1}
     \|\frac2H\divH\mathcal Q
 \bigl(Q\mathcal L_H^{-1}F\bigr)\|_{\rH^{1/2}(\T^2)}
 \leq C\|F\|_{\rH^{-1/2}(\Omega)^2}
\end{equation}
For the boundary lifting, the weak Neumann estimate gives
\begin{equation*}
 \|\mathcal N_H(G_1,G_0)\|_{\rH^1(\Omega)^2}
 \leq C\left(
 \|G_1\|_{\rH^{-1/2}(\T^2)^2}
 +\|G_0\|_{\rH^{-1/2}(\T^2)^2}\right).
\end{equation*}
However, the regularity of $\mathcal Q(Q\mathcal N_H(G_1,G_0))$ can be improved to $\rH^{3/2}$. This can be seen as follows.
Define the 
primitive $\Psi$ by
\begin{equation*}
 \Psi(x_\rH,z)=\int_z^1
 \bigl(Q\mathcal N_H(G_1,G_0)\bigr)(x_\rH,s)\,\dd s.
\end{equation*}
Consequently, we obtain
\begin{equation*}
 \partial_z\Psi=-Q\mathcal N_H(G_1,G_0),\quad
 \Psi|_{z=1}=0
 \ \text{ and }\
 \partial_z\Psi|_{z=0}
 =-\bigl[Q\mathcal N_H(G_1,G_0)\bigr]|_{z=0}
\end{equation*}
and the trace theorem therefore gives
\begin{equation*}
 \|\Psi\|_{\rH^1(\Omega)^2}
 +\|\partial_z\Psi|_{z=0}\|_{\rH^{1/2}(\T^2)^2}
 \leq C\|\mathcal N_H(G_1,G_0)\|_{\rH^1(\Omega)^2}.
\end{equation*}
To calculate the equation for $\Psi$, we define the tensor
$\mathcal C_\Psi$ by
\begin{equation*}
 \begin{aligned}
 \mathcal C_\Psi={}&\frac\mu2\left[
 \left(\int_z^1J\mathcal N_H(G_1,G_0)\,\dd s\right)
       \otimes\nablaH H
 +\nablaH H\otimes
 \left(\int_z^1J\mathcal N_H(G_1,G_0)\,\dd s\right)\right]\\
 &+\frac{\mu'-\mu}{2}
 \left(\nablaH H\cdot
 \int_z^1J\mathcal N_H(G_1,G_0)\,\dd s\right)I_2.
 \end{aligned}
\end{equation*}
Since $\mathcal L_H\mathcal N_H(G_1,G_0)=0$ and $R|_{z=1}=0$,
the product rule and integration by parts yield
\begin{equation*}
 \begin{aligned}
 \mathcal L_H\Psi={}&
 Q\Psi-\mu J\mathcal N_H(G_1,G_0)
 +\frac{\mu J}{2}\mathcal N_H(G_1,G_0)|_{z=1}
 -J\int_z^1QR\mathcal N_H(G_1,G_0)\,\dd s\\
 &-\frac J2\left(
 \int_z^1J\mathcal S_H\mathcal N_H(G_1,G_0)\,\dd s
 \right)\nablaH H
 -\divH(J\mathcal C_\Psi)
 +\partial_z(\mathcal C_\Psi\nablaH\zeta).
 \end{aligned}
\end{equation*}
Since the prescribed coefficients are smooth, we have
\begin{equation*}
 \|\mathcal C_\Psi\|_{\rH^1(\Omega)^{2\times2}}
 +\|\mathcal N_H(G_1,G_0)|_{z=1}\|_{\rL^2(\T^2)^2}
 \leq C\|\mathcal N_H(G_1,G_0)\|_{\rH^1(\Omega)^2}.
\end{equation*}
Every term in the displayed equation is therefore bounded
in $\rL^2(\Omega)^2$, and we obtain
\begin{equation*}
 \|\mathcal L_H\Psi\|_{\rL^2(\Omega)^2}
 \leq C\|\mathcal N_H(G_1,G_0)\|_{\rH^1(\Omega)^2}.
\end{equation*}
We now apply the elliptic estimate for the differential expression
$\mathcal L_H$ with zero Dirichlet data at $z=1$ and the
specified Neumann data at $z=0$. This gives
\begin{equation*}
 \begin{aligned}
 \|\Psi\|_{\rH^2(\Omega)^2}
 \leq C\left(
 \|\mathcal L_H\Psi\|_{\rL^2(\Omega)^2}
 +\|\partial_z\Psi|_{z=0}\|_{\rH^{1/2}(\T^2)^2}
 +\|\Psi\|_{\rH^1(\Omega)^2}\right)\leq C\|\mathcal N_H(G_1,G_0)\|_{\rH^1(\Omega)^2}.
 \end{aligned}
\end{equation*}
Since $\Psi|_{z=0}=\mathcal Q(Q\mathcal N_H(G_1,G_0))$,
the trace theorem and the weak Neumann estimate yield
\begin{equation*}
 \begin{aligned}
 \|\mathcal Q(Q\mathcal N_H(G_1,G_0))\|_{\rH^{3/2}(\T^2)^2}
 \leq C\left(
 \|G_1\|_{\rH^{-1/2}(\T^2)^2}
 +\|G_0\|_{\rH^{-1/2}(\T^2)^2}\right).
 \end{aligned}
\end{equation*}
The calculation is first made for smooth boundary data. Approximating
$G_1,G_0$ in $\rH^{-1/2}$, the weak estimate gives convergence of
the liftings in $\rH^1$, while the displayed estimate gives convergence
of their primitives in $\rH^2$. Thus the same conclusion holds for
all the stated boundary data. Consequently, we obtain
\begin{equation}\label{eq:basic-column-inverse2}
 \begin{aligned}
 \|\frac2H\divH\mathcal Q
 \bigl(Q\mathcal N_H(G_1,G_0)\bigr)\|_{\rH^{1/2}(\T^2)}
 \leq C\left(
 \|G_1\|_{\rH^{-1/2}(\T^2)^2}
 +\|G_0\|_{\rH^{-1/2}(\T^2)^2}\right).
 \end{aligned}
\end{equation}
Our aim is to now apply \eqref{eq:basic-column-inverse1} with $F = \mathcal G(h,\widetilde V)$ and for this purpose we estimate each term in $\mathcal G(h,\widetilde V)$ in $\rH^{-1/2}(\Omega)^2$.
For a prescribed smooth vector coefficient $\beta$ it holds that 
\begin{equation*}
 \begin{aligned}
 |\int_\Omega\beta\,\partial_{x_i}h\cdot\psi\,\dd x|
 =|\int_{\T^2}\partial_{x_i}h\,
          \mathcal Q(\beta\cdot\psi)\,\dd x_\rH|
 \leq C\|h\|_{\rH^{1/2}(\T^2)}
          \|\psi\|_{\rH^{1/2}(\Omega)^2}
 \end{aligned}
\end{equation*}
and it follows that
\begin{equation*}
 \|\beta\,\partial_{x_i}h\|_{\rH^{-1/2}(\Omega)^2}
 \leq C\|h\|_{\rH^{1/2}(\T^2)}.
\end{equation*}
After expanding the horizontal derivatives in $\mathcal G(h,\widetilde V)$,
all its remaining height terms are smooth coefficients times
$h$, $Eh$ or a first derivative of $Eh$, and hence belong to
$\rL^2(\Omega)^2$. Its velocity terms are bounded in
$\rL^2(\Omega)^2$ by $C\|\widetilde V\|_{\rH^1(\Omega)^2}$.
Using the formula for $\widetilde g_2$, we therefore obtain
\begin{equation*}
 \|\mathcal G(h,\widetilde V)\|_{\rH^{-1/2}(\Omega)^2}
 \leq C\big(
 \|h\|_{\rH^{1/2}(\T^2)}
 +\|\widetilde V\|_{\rH^1(\Omega)^2}
 +\|g_1\|_{\rH^{1/2}(\T^2)}
 +\|g_2\|_{\rL^2(\Omega)^2}\big).
\end{equation*}
At the top boundary, $\nablaH(Eh)|_{z=1}=\nablaH h$.
Recall the following bounds of the harmonic extension
\begin{equation*}
 \|Eh\|_{\rH^1(\Omega)}
 +\sum_{j=0}^1\|\partial_zEh|_{z=j}\|_{\rH^{-1/2}(\T^2)}
 \leq C\|h\|_{\rH^{1/2}(\T^2)}.
\end{equation*}
Using these bounds for the bottom correction and
$Eh|_{z=1}=h$ for the top correction gives
\begin{equation*}
 \|G_1(h)\|_{\rH^{-1/2}(\T^2)^2}
 +\|G_0(h)\|_{\rH^{-1/2}(\T^2)^2}
 \leq C\left(
 \|h\|_{\rH^{1/2}(\T^2)}
 +\|g_3\|_{\rH^{-1/2}(\T^2)^2}
 +\|g_4\|_{\rH^{-1/2}(\T^2)^2}\right).
\end{equation*}
Substituting these bounds into \eqref{eq:basic-column-inverse1}
and \eqref{eq:basic-column-inverse2} yields
\begin{equation*}
 \begin{aligned}
 &\quad\|\frac2H\divH\mathcal Q
 \bigl(Q\mathcal L_H^{-1}\mathcal G(h,\widetilde V)\bigr)
 \|_{\rH^{1/2}(\T^2)}
 +\|\frac2H\divH\mathcal Q
 \bigl(Q\mathcal N_H(G_1(h),G_0(h))\bigr)
 \|_{\rH^{1/2}(\T^2)}\\
 &\leq C\big(
 \|h\|_{\rH^{1/2}(\T^2)}
 +\|\widetilde V\|_{\rH^1(\Omega)^2}
 +\|g_1\|_{\rH^{1/2}(\T^2)}
 +\|g_2\|_{\rL^2(\Omega)^2}
 +\|g_3\|_{\rH^{-1/2}(\T^2)^2}
 +\|g_4\|_{\rH^{-1/2}(\T^2)^2}\big).
 \end{aligned}
\end{equation*}
It remains to estimate the time derivative of $B_H$.
Fix $t_\ast$ and write $Y(t)=\mathcal L_H(t)^{-1}(Q(t)\widetilde V(t_\ast))$. The input is held fixed, and all identities below are evaluated at $t=t_\ast$. Differentiating the
elliptic equation gives
\begin{equation*}
 \mathcal L_H\partial_tY
       =Q_t \widetilde V-(\partial_t\mathcal L_H)Y
       \quad\text{in }\Omega,
\end{equation*}
with boundary conditions
\begin{equation*}
 \mathcal T_H\partial_tY|_{z=1}
       =-\bigl((\partial_t\mathcal T_H)Y\bigr)|_{z=1}
 \ \text{ and }\
 \partial_z\partial_tY|_{z=0}=0.
\end{equation*}
Since $Q/(1-z)$ is smooth and bounded away from zero,
we have
\begin{equation*}
 Q_t=(1-z)\partial_t\left(\frac Q{1-z}\right)
 \ \text{ and }\ |Q_t|\leq CQ.
\end{equation*}
The smoothness of the prescribed coefficients gives
\begin{equation*}
 \begin{aligned}
 \|Q_t \widetilde V\|_{\rL^2(\Omega)^2}
 \leq C\|\sqrt Q\,\widetilde V\|_{\rL^2(\Omega)^2} \ \text{ and } \ 
 \|(\partial_t\mathcal L_H)Y\|_{\rL^2(\Omega)^2}
 +\|((\partial_t\mathcal T_H)Y)|_{z=1}
       \|_{\rH^{1/2}(\T^2)^2}
\leq C\|Y\|_{\rH^2(\Omega)^2}.
 \end{aligned}
\end{equation*}
The $\rH^2$ elliptic estimate for $Y$, followed by the
inhomogeneous conormal estimate for $\partial_tY$, therefore gives
\begin{equation*}
 \|Y\|_{\rH^2(\Omega)^2}
 +\|\partial_tY\|_{\rH^2(\Omega)^2}
 \leq C\|\sqrt Q\,\widetilde V\|_{\rL^2(\Omega)^2}.
\end{equation*}
Differentiating the definition of $B_H$ now yields
\begin{equation*}
 (\partial_tB_H)\widetilde V
 =-\frac{H_t}{H}B_H \widetilde V
       +\frac2H\divH\mathcal Q(Q_tY+Q\partial_tY),
\end{equation*}
and hence
\begin{equation}\label{eq:basic-B-time}
 \|(\partial_tB_H)\widetilde V\|_{\rH^1(\T^2)}
       \leq C\|\sqrt Q\,\widetilde V\|_{\rL^2(\Omega)^2}.
\end{equation}

We can now estimate every term in
\eqref{eq:modified-height-equation}. By
\eqref{eq:basic-column-inverse1} we obtain
\begin{equation*}
 \|B_H\nablaH h\|_{\rH^{1/2}(\T^2)}
 \leq C\|Q\nablaH h\|_{\rH^{-1/2}(\Omega)^2}
 \leq C\|h\|_{\rH^{1/2}(\T^2)}.
\end{equation*}
Moreover, \autoref{lem:height-correction-operator} and
smoothness of $b_1,b_2$ yield
\begin{equation*}
 \begin{aligned}
 \|b_1\cdot\nablaH(B_H\widetilde V)\|_{\rH^{1/2}(\T^2)}
 &\leq C\|B_H\widetilde V\|_{\rH^{3/2}(\T^2)}
 \leq C\|\widetilde V\|_{\rH^{1/2}(\Omega)^2},\\
 \|b_2B_H \widetilde V\|_{\rH^{1/2}(\T^2)}
 &\leq C\|B_H\widetilde V\|_{\rH^1(\T^2)}
 \leq C\|\sqrt Q\,\widetilde V\|_{\rL^2(\Omega)^2}.
 \end{aligned}
\end{equation*}
Combining these estimates with
\eqref{eq:basic-B-time} and
$\rH^1(\Omega)\hookrightarrow\rH^{1/2}(\Omega)$ gives
\begin{equation}\label{eq:modified-height-source-bound}
 \begin{aligned}
 \|\widetilde h_t+b_1\cdot\nablaH\widetilde h
                    +b_2\widetilde h\|_{\rH^{1/2}(\T^2)}
 \leq C\big(
       \|h\|_{\rH^{1/2}(\T^2)}
       +\|\widetilde V\|_{\rH^1(\Omega)^2}
       +\mathfrak g(t)\big).
 \end{aligned}
\end{equation}
where the source norm $\mathfrak g(t)$ is given by
\begin{equation*}
 \mathfrak g(t)=\big(
 \|g_1(t)\|_{\rH^{1/2}(\T^2)}^2
 +\|g_2(t)\|_{\rL^2(\Omega)^2}^2
 +\|g_3(t)\|_{\rH^{-1/2}(\T^2)^2}^2
 +\|g_4(t)\|_{\rH^{-1/2}(\T^2)^2}^2
 \big)^{1/2}.
\end{equation*}
As in the toy model, differentiating the fractional norm gives
\begin{equation*}
 \begin{aligned}
 \frac12\frac{\dd}{\dd t}
       \|\widetilde h\|_{\rH^{1/2}(\T^2)}^2
 ={}&\int_{\T^2}(1-\DeltaH)^{1/4}\widetilde h\,
       (1-\DeltaH)^{1/4}
       \bigl(\widetilde h_t+b_1\cdot\nablaH\widetilde h
                           +b_2\widetilde h\bigr)\,\dd x_\rH\\
 &-\int_{\T^2}(1-\DeltaH)^{1/4}\widetilde h\,
       (1-\DeltaH)^{1/4}
       \bigl(b_1\cdot\nablaH\widetilde h
                           +b_2\widetilde h\bigr)\,\dd x_\rH.
 \end{aligned}
\end{equation*}
H\"older's inequality estimates the first integral.
For the second, the exact same commutator calculation from the toy
model yields
\begin{equation*}
 \begin{aligned}
 &\|\bigl[(1-\DeltaH)^{1/4},
                 b_1\cdot\nablaH\bigr]\widetilde h
       \|_{\rL^2(\T^2)}
 +\|(1-\DeltaH)^{1/4}(b_2\widetilde h)\|_{\rL^2(\T^2)}
 \leq C\|\widetilde h\|_{\rH^{1/2}(\T^2)}.
 \end{aligned}
\end{equation*}
The remaining transport contribution satisfies
\begin{equation*}
 \int_{\T^2}(1-\DeltaH)^{1/4}\widetilde h\,
       b_1\cdot\nablaH(1-\DeltaH)^{1/4}\widetilde h
       \,\dd x_\rH
 =-\frac12\int_{\T^2}(\divH b_1)
       |(1-\DeltaH)^{1/4}\widetilde h|^2\,\dd x_\rH.
\end{equation*}
Using \eqref{eq:modified-height-source-bound} and Young's
inequality, we consequently obtain, for every $\varepsilon>0$, that
\begin{equation}\label{eq:modified-height-fractional-energy}
 \begin{aligned}
 \frac12\frac{\dd}{\dd t}
       \|\widetilde h\|_{\rH^{1/2}(\T^2)}^2
 \leq\varepsilon\|\widetilde V\|_{\rH^1(\Omega)^2}^2
 &+C_\varepsilon\big(
       \|\widetilde h\|_{\rH^{1/2}(\T^2)}^2
       +\|h\|_{\rH^{1/2}(\T^2)}^2
       +\mathfrak g(t)^2\big).
 \end{aligned}
\end{equation}

We next test momentum in \eqref{eq:off-exact-momentum} with $\widetilde V$ and the height
equation with $(H/2)h$. The latter factor cancels the
pressure and column contributions because
\begin{equation*}
 \int_{\T^2}h\,\divH\mathcal Q(Q\widetilde V)\,\dd x_\rH
       =-\int_\Omega Q\nablaH h\cdot\widetilde V\,\dd x.
\end{equation*}
To describe the viscous contribution, we define
$\mathcal D_H(U)$ by
\begin{equation*}
 \mathcal D_H(U)=\int_\Omega
       \left[J\operatorname{tr}\bigl(
                 (\mathcal S_HU)^\top\nablaH U\bigr)
             +\mathcal T_HU\cdot U_z\right]\,\dd x.
\end{equation*}
The coercivity of the shifted operator $\mathcal L_H$ gives
\begin{equation}\label{eq:basic-viscous-coercivity}
 \mathcal D_H(U)+\|\sqrt Q\,U\|_{\rL^2(\Omega)^2}^2
       \geq c\|U\|_{\rH^1(\Omega)^2}^2.
\end{equation}
The shift $QU$ in $\mathcal L_HU$ cancels the identical term
in $\mathcal G(h,U)$. Moreover, the definition of $L$ gives
\begin{equation*}
 \mathcal B|_{z=0}=0
 \ \text{ and }\
 \mathcal B|_{z=1}=-\frac H2L-\divH\mathcal Q(QV)=0.
\end{equation*}
Integration by parts therefore gives
\begin{equation*}
 \begin{aligned}
 &\quad\frac12\frac{\dd}{\dd t}\left(
       \int_\Omega Q|\widetilde V|^2\,\dd x
       +\frac12\int_{\T^2}Hh^2\,\dd x_\rH\right)
       +\mathcal D_H(\widetilde V)\\
 &=\frac12\int_\Omega
       \bigl(Q_t+\divH(QV)+\partial_z\mathcal B\bigr)
                         |\widetilde V|^2\,\dd x
       +\frac12\int_{\T^2}Hg_1h\,\dd x_\rH+\frac14\int_{\T^2}
       \bigl(H_t-\divH\mathcal Q(JV)-2L\bigr)h^2\,\dd x_\rH\\
 &\quad+\int_\Omega\left[
       \mathcal G(h,\widetilde V)-Q\widetilde V
       +QV\cdot\nablaH\widetilde V+\mathcal B\widetilde V_z
       \right]\cdot\widetilde V\,\dd x+\int_{\T^2}G_1(h)\cdot\widetilde V|_{z=1}\,\dd x_\rH
      \\& \quad -\int_{\T^2}\frac{\mu}{J|_{z=0}}
               G_0(h)\cdot\widetilde V|_{z=0}\,\dd x_\rH.
 \end{aligned}
\end{equation*}
The mass defect satisfies
\begin{equation*}
 Q_t+\divH(QV)+\partial_z\mathcal B
       =\delta_HQ[H_t-L],
 \quad
 \left|Q_t+\divH(QV)+\partial_z\mathcal B\right|\leq CQ.
\end{equation*}
Consequently, its integral is bounded by
$C\|\sqrt Q\,\widetilde V\|_{\rL^2(\Omega)^2}^2$.
The two surface integrals involving $h$ and $g_1$ satisfy
\begin{equation*}
 \begin{aligned}
 |\frac12\int_{\T^2}Hg_1h\,\dd x_\rH|
 +|\frac14\int_{\T^2}
       \bigl(H_t-\divH\mathcal Q(JV)-2L\bigr)h^2\,\dd x_\rH|
\leq C\big(
       \|h\|_{\rL^2(\T^2)}^2
       +\|g_1\|_{\rL^2(\T^2)}^2\big).
 \end{aligned}
\end{equation*}

For the momentum integral, we use the previously established
$\rH^{-1/2}(\Omega)^2$ estimate for $\mathcal G$.
Smoothness of the prescribed coefficients gives
\begin{equation*}
 \|Q\widetilde V\|_{\rL^2(\Omega)^2}
 +\|QV\cdot\nablaH\widetilde V\|_{\rL^2(\Omega)^2}
 +\|\mathcal B\partial_z\widetilde V\|_{\rL^2(\Omega)^2}
 \leq C\|\widetilde V\|_{\rH^1(\Omega)^2}.
\end{equation*}
Since $\rL^2(\Omega)^2\hookrightarrow\rH^{-1/2}(\Omega)^2$,
we obtain
\begin{equation*}
 \begin{aligned}
 |\int_\Omega\left[
       \mathcal G(h,\widetilde V)-Q\widetilde V
       +QV\cdot\nablaH\widetilde V
       +\mathcal B\partial_z\widetilde V
       \right]\cdot\widetilde V\,\dd x|
\quad\leq C\big(
       \|h\|_{\rH^{1/2}(\T^2)}
       +\|\widetilde V\|_{\rH^1(\Omega)^2}
       +\mathfrak g(t)\big)
       \|\widetilde V\|_{\rH^{1/2}(\Omega)^2}.
 \end{aligned}
\end{equation*}
To absorb the velocity contribution, we use the embeddings
\begin{equation*}
 \rH^1(\Omega)^2\Subset\rH^{1/2}(\Omega)^2
 \hookrightarrow\rL^2(\Omega,Q\,\dd x)^2.
\end{equation*}
The second embedding is continuous and injective because
$Q$ is bounded and positive in the interior. Therefore, for every $\eta>0$, it holds that
\begin{equation*}
 \|\widetilde V\|_{\rH^{1/2}(\Omega)^2}
 \leq\eta\|\widetilde V\|_{\rH^1(\Omega)^2}
       +C_\eta\|\sqrt Q\,\widetilde V\|_{\rL^2(\Omega)^2}.
\end{equation*}
The constants are uniform under the prescribed comparison
bounds between $Q$ and $1-z$. Choosing $\eta$ sufficiently
small and then applying Young's inequality, we obtain,
for every $\varepsilon>0$,
\begin{equation*}
 \begin{aligned}
 &\quad|\int_\Omega\left[
       \mathcal G(h,\widetilde V)-Q\widetilde V
       +QV\cdot\nablaH\widetilde V
       +\mathcal B\partial_z\widetilde V
       \right]\cdot\widetilde V\,\dd x|\\
 &\leq\varepsilon\|\widetilde V\|_{\rH^1(\Omega)^2}^2
       +C_\varepsilon\big(
       \|\sqrt Q\,\widetilde V\|_{\rL^2(\Omega)^2}^2
       +\|h\|_{\rH^{1/2}(\T^2)}^2
       +\mathfrak g(t)^2\big).
 \end{aligned}
\end{equation*}
Next, 
the trace theorem and the bounds for $G_1(h)$ and $G_0(h)$
likewise give
\begin{equation*}
 \begin{aligned}
 |\int_{\T^2}G_1(h)\cdot \widetilde V|_{z=1}\,\dd x_\rH|
 +|\int_{\T^2}\frac{\mu}{J|_{z=0}}
                     G_0(h)\cdot \widetilde V|_{z=0}\,\dd x_\rH|
 \leq C\big(
       \|h\|_{\rH^{1/2}(\T^2)}
       +\mathfrak g(t)\big)\|\widetilde V\|_{\rH^1(\Omega)^2}.
 \end{aligned}
\end{equation*}
Combining these estimates with Young's inequality and
\eqref{eq:basic-viscous-coercivity}, we obtain
\begin{equation}\label{eq:basic-height-velocity-energy}
 \begin{aligned}
 \frac12\frac{\dd}{\dd t}\left(
       \int_\Omega Q|\widetilde V|^2\,\dd x
       +\frac12\int_{\T^2}Hh^2\,\dd x_\rH\right)
       +c\|\widetilde V\|_{\rH^1(\Omega)^2}^2
\leq C\big(
       \|\sqrt Q\,\widetilde V\|_{\rL^2(\Omega)^2}^2
       +\|h\|_{\rH^{1/2}(\T^2)}^2
       +\mathfrak g(t)^2\big).
 \end{aligned}
\end{equation}

We now add the fractional height estimate and define the
combined energy $E$ by
\begin{equation*}
 E(t)=\frac12\int_\Omega Q|\widetilde V|^2\,\dd x
       +\frac14\int_{\T^2}Hh^2\,\dd x_\rH
       +\frac12\|\widetilde h\|_{\rH^{1/2}(\T^2)}^2.
\end{equation*}
Since $h=\widetilde h+B_H\widetilde V$,
\autoref{lem:height-correction-operator} gives
\begin{equation*}
 \begin{aligned}
 \frac{1}{2}\|h(t)\|_{\rH^{1/2}(\T^2)}^2
 \leq \|\widetilde h(t)\|_{\rH^{1/2}(\T^2)}^2
       +\|B_H\widetilde V(t)\|_{\rH^{1/2}(\T^2)}^2\\
 \leq C\big(
       \|\widetilde h(t)\|_{\rH^{1/2}(\T^2)}^2
       +\|\sqrt{Q(t)}\,\widetilde V(t)\|_{\rL^2(\Omega)^2}^2
       \big)
 \leq CE(t).
 \end{aligned}
\end{equation*}
Choosing $\varepsilon$ sufficiently small in
\eqref{eq:modified-height-fractional-energy} and adding
\eqref{eq:basic-height-velocity-energy}, we therefore obtain
\begin{equation*}
 E'(t)+c\|\widetilde V(t)\|_{\rH^1(\Omega)^2}^2
       \leq CE(t)+C\mathfrak g(t)^2.
\end{equation*}
Since the corrections have zero initial values, $E(0)=0$.
The preceding calculations and Gronwall's inequality give
the following basic energy estimate.

\begin{lem}[Basic energy estimate]\label{lem:basic-energy}
Let $(H,V)$ be a prescribed smooth background with $H(0)=H_0$, choose $T>0$ sufficiently small that $\phi(\cdot,t)$ is well defined and let $(h,\widetilde V)$
be a smooth solution of \eqref{eq:off-exact-momentum}
subject to \eqref{eq:good-boundary}, with zero initial values.
Then
\begin{equation*}
 \begin{aligned}
 \sup_{0\leq t\leq T}\left(
       \|h(t)\|_{\rH^{1/2}(\T^2)}^2
       +\|\sqrt{Q(t)}\,\widetilde V(t)
                              \|_{\rL^2(\Omega)^2}^2\right)
       +\int_0^T\|\widetilde V(t)\|_{\rH^1(\Omega)^2}^2\,\dd t\leq C_T\int_0^T\mathfrak g(t)^2\,\dd t.
 \end{aligned}
\end{equation*}
The constant $C_T$ depends on $T$, finitely many prescribed
background norms, the viscosity coefficients and the
positive chart bounds.
\end{lem}

\subsection{Higher order tame estimates}

To obtain higher order estimates, we differentiate in time and the horizontal variables. We use smooth
sources and corrections whose initial time derivatives vanish at every
order. For $m\in\N_0$ and
$i=(i_0,i_1,i_2)\in\N_0^3$ with $|i|\leq m$, we define
\begin{equation*}
 \partial^i=\partial_t^{i_0}\partial_{x_1}^{i_1}\partial_{x_2}^{i_2},
 \quad h^i=\partial^ih,\quad
 \widetilde V^{\,i}=\partial^i\widetilde V
 \ \text{ and }\ g_j^i=\partial^ig_j\quad(j=1,2,3,4).
\end{equation*}
The superscripts denote time and horizontal derivatives. All coefficients
are evaluated at the same prescribed smooth background $(H,V)$.
Collecting the coefficients in \eqref{eq:off-exact-momentum}
as in the basic energy calculation and differentiating gives
\begin{equation}\label{eq:higher-system}
 \left\{\begin{aligned}
 \partial_th^i+b_1\cdot\nablaH h^i+b_2h^i
       +\frac2H\divH\mathcal Q(Q\widetilde V^{\,i})
 &=g_1^i+r_1^i &&\text{in }(0,T)\times\T^2,\\
 Q(\partial_t\widetilde V^{\,i}+\nablaH h^i)
       -\divH(J\mathcal S_H\widetilde V^{\,i})
       -\partial_z\mathcal T_H\widetilde V^{\,i}
 &=g_2^i+r_2^i &&\text{in }(0,T)\times\Omega,\\
 h^i(0)=0\ \text{ and }\ \widetilde V^{\,i}(0)&=0.
 \end{aligned}\right.
\end{equation}
Differentiating \eqref{eq:good-boundary} gives
\begin{equation}\label{eq:higher-boundary}
 \mathcal T_H\widetilde V^{\,i}|_{z=1}=g_3^i+r_3^i
 \ \text{ and }\
 \partial_z\widetilde V^{\,i}|_{z=0}=g_4^i+r_4^i.
\end{equation}
The quantities $r_j^i$ contain the coefficient commutators and the
remaining linear terms on the right-hand side of the good-unknown
equations. A direct calculation yields
\begin{equation*}
 \begin{aligned}
 r_1^i={}&-\sum_{0<j\leq i}\binom ij
       \left((\partial^jb_1)\cdot\nablaH h^{i-j}
                    +(\partial^jb_2)h^{i-j}\right)\\
 &-\sum_{\substack{j,k\geq0,\ j+k\leq i\\|j|+|k|>0}}
       \binom ij\binom{i-j}{k}
       \partial^j\left(\frac2H\right)
       \divH\mathcal Q\left(
          (\partial^kQ)\widetilde V^{\,i-j-k}\right).
 \end{aligned}
\end{equation*}
To write the differentiated momentum terms compactly, we define
the vector field $\mathcal M^i$ by
\begin{equation*}
 \mathcal M^i
 =\sum_{j\leq i}\binom ij
       \left((\partial^jQ)\widetilde V^{\,i-j}
                +\frac12\partial^j(JV)h^{i-j}\right)
 =\partial^i\left(Q\widetilde V+\frac{Jh}{2}V\right).
\end{equation*}
For the differentiated viscosity, we define the matrix
$\Pi_\rH^i$ and the vector $\Pi_z^i$ by
\begin{equation*}
 \Pi_\rH^i=\partial^i(J\mathcal S_H\widetilde V)
                   -J\mathcal S_H\widetilde V^{\,i}
 \ \text{ and }\
 \Pi_z^i=\partial^i(\mathcal T_H\widetilde V)
                   -\mathcal T_H\widetilde V^{\,i}.
\end{equation*}
Using \eqref{eq:good-stress-flux}, their full expressions are
\begin{equation*}
 \begin{aligned}
 \Pi_\rH^i
& =\sum_{0<j\leq i}\binom ij\biggl[
 \mu(\partial^jJ)
       \left(\nablaH\widetilde V^{\,i-j}
                    +(\nablaH\widetilde V^{\,i-j})^\top\right)
 +(\mu'-\mu)(\partial^jJ)
                    (\divH\widetilde V^{\,i-j})I_2\\
 &\quad -\mu\left(
       \partial_z\widetilde V^{\,i-j}\otimes\nablaH\partial^j\zeta
       +\nablaH\partial^j\zeta\otimes\partial_z\widetilde V^{\,i-j}
       \right)-(\mu'-\mu)
       \left(\nablaH\partial^j\zeta
                    \cdot\partial_z\widetilde V^{\,i-j}\right)I_2
 \biggr]
 \end{aligned}
\end{equation*}
and
\begin{equation*}
 \begin{aligned}
 \Pi_z^i
 &=\sum_{0<j\leq i}\binom ij\biggl[
 \mu\partial^j\left(\frac{1+|\nablaH\zeta|^2}{J}\right)
                     \partial_z\widetilde V^{\,i-j}
 +\mu'\partial^j\left(
       \frac{\nablaH\zeta\otimes\nablaH\zeta}{J}\right)
                     \partial_z\widetilde V^{\,i-j}\\
 &\quad-\mu\left(\nablaH\widetilde V^{\,i-j}
                    +(\nablaH\widetilde V^{\,i-j})^\top\right)
                     \nablaH\partial^j\zeta
 -(\mu'-\mu)(\divH\widetilde V^{\,i-j})
                     \nablaH\partial^j\zeta
 \biggr].
 \end{aligned}
\end{equation*}
Consequently, we obtain
\begin{equation*}
 \begin{aligned}
 r_2^i={}&\divH\Pi_\rH^i+\partial_z\Pi_z^i
       -\sum_{0<j\leq i}\binom ij(\partial^jQ)
           \left(\partial_t\widetilde V^{\,i-j}+\nablaH h^{i-j}\right)\\
 &-\sum_{j\leq i}\binom ij\biggl[
       \partial^j\left(\frac QJ V_z\right)
                                 Eg_1^{i-j}
       +(\partial^j(QV))\cdot\nablaH\widetilde V^{\,i-j}
       +\partial^j(Q\nablaH V)\widetilde V^{\,i-j}
       +(\partial^j\mathcal B)\partial_z\widetilde V^{\,i-j}\\
 &
       +\frac12\partial^j\left[
           J(V_t+V\cdot\nablaH V+\nablaH H)
                +\left(\frac{\zeta L}{H}-EL\right)V_z
                                  \right]h^{i-j}\\
 &
       +\partial^j\left(\frac\zeta H V_z\right)
                                  \divH\mathcal Q\mathcal M^{i-j}
       -(\partial^jV_z)\divH
             \int_0^z\mathcal M^{i-j}(x_\rH,s,t)\,\dd s
       \biggr]\\
 &+\sum_{j\leq i}\binom ij
       \partial^j\left[
         \frac Q{J^2}\partial_z\bigl(E\mathcal F_H(H,V)\bigr)V_z
                         \right]Eh^{i-j}-\partial_z\sum_{j\leq i}\binom ij
       \partial^j\left(\frac{\mathcal F_V(H,V)}J\right)Eh^{i-j}.
 \end{aligned}
\end{equation*}
At the top boundary, $Eh|_{z=1}=h$ gives
\begin{equation*}
 G_1(h)=g_3+\left[
       (\mathcal S_HV)\nablaH h
       -\frac{(\mathcal S_HV)\nablaH J
                         +\partial_z\mathcal T_HV}{J}h
       \right]|_{z=1}.
\end{equation*}
Differentiating this identity and the bottom condition yields
\begin{equation*}
 \begin{aligned}
 r_3^i=-\Pi_z^i|_{z=1}+\bigg[\sum_{j\leq i}\binom ij
       \left(
       \partial^j(\mathcal S_HV)\nablaH h^{i-j}
       -\partial^j\left(
          \frac{(\mathcal S_HV)\nablaH J
                          +\partial_z\mathcal T_HV}{J}\right)h^{i-j}
       \right)\bigg]|_{z=1}
 \end{aligned}
\end{equation*}
and
\begin{equation*}
 r_4^i=-\bigg[\sum_{j\leq i}\binom ij
          \partial^j\left(\frac{V_z}{J}\right)
                              \partial_zEh^{i-j}\bigg]|_{z=0}.
\end{equation*}
The term $-\Pi_z^i|_{z=1}$ comes from the same differentiated
viscous flux as $\partial_z\Pi_z^i$ in momentum.
For $i=0$, we have $g_j^0=g_j$, and these definitions recover the
original equations because
\begin{equation*}
 r_1^0=0,\quad
 r_2^0=\mathcal G(h,\widetilde V)-Q\widetilde V-g_2,\quad
 r_3^0=G_1(h)-g_3
 \ \text{ and }\ r_4^0=G_0(h)-g_4.
\end{equation*}

As in the basic energy estimate, we seek
$\rH^{1/2}(\T^2)$ control of each $h^i$.
We therefore define the modified differentiated height by
\begin{equation*}
 \widetilde h^{\,i}=h^i-B_H\widetilde V^{\,i}.
\end{equation*}
Since $B_H$ depends on time and the horizontal variables through the
background, its relation to the derivative of $\widetilde h$ is
\begin{equation*}
 \widetilde h^{\,i}
 =\partial^i\widetilde h
       +\partial^i(B_H\widetilde V)
       -B_H\partial^i\widetilde V.
\end{equation*}
We apply the elliptic decomposition directly to
\eqref{eq:higher-system} with \eqref{eq:higher-boundary}.
Because $\mathcal L_H$ contains the shift $Q$, this gives
\begin{equation*}
 \widetilde V^{\,i}
 =\mathcal L_H^{-1}\big(
       g_2^i+r_2^i+Q\widetilde V^{\,i}
       -Q\nablaH h^i-Q\partial_t\widetilde V^{\,i}\big)
       +\mathcal N_H(g_3^i+r_3^i,g_4^i+r_4^i).
\end{equation*}
Substituting this identity into the height equation and using
\begin{equation*}
 \partial_t\widetilde h^{\,i}
 =\partial_th^i-B_H\partial_t\widetilde V^{\,i}
                    -(\partial_tB_H)\widetilde V^{\,i}
\end{equation*}
therefore yields
\begin{equation}\label{eq:higher-modified-height}
 \begin{aligned}
 \partial_t\widetilde h^{\,i}
       +b_1\cdot\nablaH\widetilde h^{\,i}+b_2\widetilde h^{\,i}
 ={}&g_1^i+r_1^i+B_H\nablaH h^i
       -(\partial_tB_H)\widetilde V^{\,i}
       -b_1\cdot\nablaH(B_H\widetilde V^{\,i})
       -b_2B_H\widetilde V^{\,i}\\
 &-\frac2H\divH\mathcal Q\left(
       Q\mathcal L_H^{-1}(g_2^i+r_2^i+Q\widetilde V^{\,i}) +
       Q\mathcal N_H(g_3^i+r_3^i,g_4^i+r_4^i)\right),
 \end{aligned}
\end{equation}
with $h^{\,i}(0)=0$.
Here, $\partial_tB_H$ denotes
the derivative of the operator with its input held fixed as in the energy tame derivation.
The differentiated viscous terms must be estimated together with
their boundary terms.  To explain the required estimate, suppose
that $Y$ satisfies
\begin{equation*}
 \mathcal L_HY=F+\divH\Pi_\rH+\partial_z\Pi_z
 \quad\text{in }\Omega,
\end{equation*}
with
\begin{equation*}
 (\mathcal T_HY+\Pi_z)|_{z=1}=\eta_1
 \ \text{ and }\
 (\mathcal T_HY+\Pi_z)|_{z=0}=\eta_0,
\end{equation*}
in the weak sense.
Coercivity of $\mathcal{L}_H$ gives
\begin{equation*}
 \|Y\|_{\rH^1(\Omega)^2}\leq C\bigl(
 \|F\|_{\rH^{-1/2}(\Omega)^2}
 +\|\Pi_\rH\|_{\rL^2(\Omega)^{2\times2}}
 +\|\Pi_z\|_{\rL^2(\Omega)^2}
 +\|\eta_1\|_{\rH^{-1/2}(\T^2)^2}
 +\|\eta_0\|_{\rH^{-1/2}(\T^2)^2}\bigr).
\end{equation*}

To derive an estimate for $\mathcal Q(QY)$, we test the weak
equation with $R\varphi(x_\rH)$, where
$\varphi\in\rC^\infty(\T^2)^2$. Since $JR=Q$, this produces
the weighted horizontal stress $\mathcal Q(Q\mathcal S_HY)$.
Moreover, we have
\begin{equation*}
 \partial_{x_a}Q-\partial_z(R\partial_{x_a}\zeta)
       =\frac J2\partial_{x_a}H,\quad a=1,2,
 \quad R|_{z=1}=0
 \ \text{ and }\ R|_{z=0}=\frac H2.
\end{equation*}
To express the derivatives of the weighted average, we define
$\mathcal E_Y$ by
\begin{equation*}
 \mathcal E_Y=\frac12\mathcal Q(JY)\otimes\nablaH H.
\end{equation*}
Integration by parts then gives
\begin{equation*}
 \mathcal Q\left(
 Q\left[\nablaH Y-\frac{Y_z\otimes\nablaH\zeta}{J}\right]\right)
 =\nablaH\mathcal Q(QY)-\mathcal E_Y.
\end{equation*}
Consequently, testing with $R\varphi$ yields the following
equation on $\T^2$ in the sense of distributions,
\begin{equation*}
 \begin{aligned}
 -\mu\DeltaH\mathcal Q(QY)-\mu'\nablaH\divH\mathcal Q(QY)
 ={}&\mathcal Q(RF)+\divH\mathcal Q(R\Pi_\rH)
       -\mathcal Q(\Pi_\rH\nablaH R+\Pi_z\partial_zR)
       -\frac H2\eta_0\\
 &-\mathcal Q\left(
       \frac J2(\mathcal S_HY)\nablaH H
                           -\frac\mu2Y_z+QRY\right)\\
 &-\divH\left[
       \mu(\mathcal E_Y+\mathcal E_Y^\top)
       +(\mu'-\mu)\operatorname{tr}(\mathcal E_Y)I_2\right].
 \end{aligned}
\end{equation*}
The top boundary contribution vanishes because $R|_{z=1}=0$.
Here we estimate
\begin{equation*}
 \begin{aligned}
 |\int_{\T^2}\mathcal Q(RF)\cdot\varphi\,\dd x_\rH|
 =|\int_\Omega F\cdot R\varphi\,\dd x|
 \leq \|F\|_{\rH^{-1/2}(\Omega)^2}
              \|R\varphi\|_{\rH^{1/2}(\Omega)^2}
 \leq C\|F\|_{\rH^{-1/2}(\Omega)^2}
              \|\varphi\|_{\rH^{1/2}(\T^2)^2}
 \end{aligned}
\end{equation*}
so that 
$ \|\mathcal Q(RF)\|_{\rH^{-1/2}(\T^2)^2}
       \leq C\|F\|_{\rH^{-1/2}(\Omega)^2}.$
The terms involving the viscous commutators satisfy
\begin{equation*}
 \begin{aligned}
 \|\divH\mathcal Q(R\Pi_\rH)\|_{\rH^{-1/2}(\T^2)^2}
 &\leq C\|\Pi_\rH\|_{\rL^2((0,1),\rH^{1/2}(\T^2)^{2\times2})},\\
 \|\mathcal Q(\Pi_\rH\nablaH R+\Pi_z\partial_zR)
                         \|_{\rL^2(\T^2)^2}
 &\leq C\bigl(\|\Pi_\rH\|_{\rL^2(\Omega)^{2\times2}}
                         +\|\Pi_z\|_{\rL^2(\Omega)^2}\bigr).
 \end{aligned}
\end{equation*}
For the remaining terms, smoothness of the prescribed coefficients
gives
\begin{equation*}
 \begin{aligned}
 \|\mathcal Q\big(
       \frac J2(\mathcal S_HY)\nablaH H
              -\frac\mu2Y_z+QRY\big)\|_{\rL^2(\T^2)^2}
 &\leq C\|Y\|_{\rH^1(\Omega)^2},\\
 \left\|\divH\left[
       \mu(\mathcal E_Y+\mathcal E_Y^\top)
       +(\mu'-\mu)\operatorname{tr}(\mathcal E_Y)I_2
       \right]\right\|_{\rL^2(\T^2)^2}
 &\leq C\|\mathcal E_Y\|_{\rH^1(\T^2)^{2\times2}}
 \leq C\|Y\|_{\rH^1(\Omega)^2}\ \text{ and }\\
 \|\frac H2\eta_0\|_{\rH^{-1/2}(\T^2)^2}
 &\leq C\|\eta_0\|_{\rH^{-1/2}(\T^2)^2}.
 \end{aligned}
\end{equation*}
The operator $-\mu\DeltaH-\mu'\nablaH\divH$ is the periodic two dimensional
Lam\'e operator on $\T^2$ which satisfies the estimate
\begin{equation*}
 \begin{aligned}
 \|\mathcal Q(QY)\|_{\rH^{3/2}(\T^2)^2}
 \leq C\bigl(
 \|-\mu\DeltaH\mathcal Q(QY)
       -\mu'\nablaH\divH\mathcal Q(QY)\|_{\rH^{-1/2}(\T^2)^2}+\|\mathcal Q(QY)\|_{\rL^2(\T^2)^2}\bigr).
 \end{aligned}
\end{equation*}
Since
\begin{equation*}
 \|\mathcal Q(QY)\|_{\rL^2(\T^2)^2}
       \leq C\|Y\|_{\rH^1(\Omega)^2}
 \ \text{ and }\
 \|\frac2H\divH\mathcal Q(QY)\|_{\rH^{1/2}(\T^2)}
       \leq C\|\mathcal Q(QY)\|_{\rH^{3/2}(\T^2)^2},
\end{equation*}
combining the preceding estimates with the weak estimate for $Y$
and $\rL^2(\T^2)\hookrightarrow\rH^{-1/2}(\T^2)$ gives
\begin{equation}\label{eq:structured-averaged-inverse}
 \begin{aligned}
&\quad \|Y\|_{\rH^1(\Omega)^2}
 +\|\frac2H\divH\mathcal Q(QY)\|_{\rH^{1/2}(\T^2)}
\\& \leq C\bigl(\|F\|_{\rH^{-1/2}(\Omega)^2}
 +\|\Pi_\rH\|_{\rL^2((0,1),\rH^{1/2}(\T^2)^{2\times2})}
 +\|\Pi_z\|_{\rL^2(\Omega)^2}+\|\eta_1\|_{\rH^{-1/2}(\T^2)^2}
 +\|\eta_0\|_{\rH^{-1/2}(\T^2)^2}\bigr).
 \end{aligned}
\end{equation}
We now apply this estimate to the
elliptic terms in \eqref{eq:higher-modified-height} by  setting $Y_i$ and $F_i$ to be 
\begin{equation*}
 \begin{aligned}
 Y_i=\mathcal L_H^{-1}
       \bigl(g_2^i+r_2^i+Q\widetilde V^{\,i}\bigr)
       +\mathcal N_H(g_3^i+r_3^i,g_4^i+r_4^i)
 \ \text{ and }\
 F_i=g_2^i+r_2^i+Q\widetilde V^{\,i}
       -\divH\Pi_\rH^i-\partial_z\Pi_z^i.
 \end{aligned}
\end{equation*}
Thus $F_i$ contains the remaining momentum terms after removing
the two viscous divergences, while $Q\widetilde V^{\,i}$
comes from the shift in $\mathcal L_H$.
The corresponding combined boundary sources are given by
\begin{equation*}
 \begin{aligned}
 \eta_{1,i}
 =g_3^i+r_3^i+\Pi_z^i|_{z=1}
   =\partial^iG_1(h)
 \ \text{ and }\
 \eta_{0,i}
 =\frac{\mu}{J|_{z=0}}(g_4^i+r_4^i)+\Pi_z^i|_{z=0}
   =\partial^i\left[\frac{\mu}{J|_{z=0}}G_0(h)\right].
 \end{aligned}
\end{equation*}
We now estimate all sources in their corresponding norms.
We define the differentiated energy and the higher-order source norm by
\begin{equation*}
 \begin{aligned}
 E_i&=\frac12\int_\Omega Q|\widetilde V^{\,i}|^2\,\dd x
 +\frac14\int_{\T^2}H|h^i|^2\,\dd x_\rH
 +\frac12\|\widetilde h^{\,i}\|_{\rH^{1/2}(\T^2)}^2
 \ \text{ and }\\
 \mathfrak g_m(t)^2&=\sum_{|i|\leq m}\bigg(
 \|g_1^i\|_{\rH^{1/2}(\T^2)}^2
 +\|g_2^i\|_{\rL^2(\Omega)^2}^2
 +\sum_{a=3}^4\|g_a^i\|_{\rH^{-1/2}(\T^2)^2}^2\bigg).
 \end{aligned}
\end{equation*}
Thus $E_0$ is the basic energy and $\mathfrak g_0=\mathfrak g$.
The identity $h^i=\widetilde h^{\,i}+B_H\widetilde V^{\,i}$ gives
\begin{equation*}
 \|h^i\|_{\rH^{1/2}(\T^2)}^2
 \leq2\|\widetilde h^{\,i}\|_{\rH^{1/2}(\T^2)}^2
 +C\|\sqrt Q\,\widetilde V^{\,i}\|_{\rL^2(\Omega)^2}^2
 \leq CE_i.
\end{equation*}
To retain the background dependence, fix sufficiently large integers
$m_0\geq2$ and $d$ and write
\begin{equation*}
 C_m=C_m(H,V)=1+\|H\|_{\rH^{m+d}((0,T)\times\T^2)}
 +\|V\|_{\rH^{m+d}((0,T)\times\Omega)^2}.
\end{equation*}
Below, $c_m$ depends on $m$, the fixed time interval, the viscosities,
the positive chart bounds and $C_{m_0}$, but not on higher background
norms. The fixed margin $d$ includes the derivatives in the
coefficients and the Sobolev embeddings and traces used in their
estimates.
Since the background is smooth, multiplication by its coefficients
is bounded on the Sobolev spaces under consideration. Using the
harmonic extension and normal-trace estimates, we obtain
\begin{equation*}
 \|Eh^j\|_{\rH^1(\Omega)}
 +\|\nablaH h^j\|_{\rH^{-1/2}(\T^2)^2}
 +\sum_{a=0}^1
       \|\partial_zEh^j|_{z=a}\|_{\rH^{-1/2}(\T^2)}
 \leq C\sqrt{E_j}.
\end{equation*}
These bounds control the height terms in $F_i$ and $\eta_{a,i}$.
The other velocity
terms in $F_i$ are controlled by
$\|\widetilde V^{\,j}\|_{\rH^1}$.
For the mass commutator, the identity
$\partial^jQ=(1-z)\partial^j(Q/(1-z))$ gives
\begin{equation*}
 \|(\partial^jQ)\partial_t\widetilde V^{\,i-j}
                         \|_{\rL^2(\Omega)^2}
 \leq c_m C_{|j|}
       \|\sqrt Q\,\partial_t\widetilde V^{\,i-j}
                         \|_{\rL^2(\Omega)^2},
 \quad 0<j\leq i.
\end{equation*}
The unknown derivative has tangent order at most $|i|$.
In every remaining product, $k$ derivatives on a background
coefficient leave at most $m-k+1$ tangent derivatives on the unknown.
The explicit formulas for the sources therefore yield
\begin{equation*}
 \begin{aligned}
 &\quad\sum_{|i|\leq m}\big(
       \|F_i\|_{\rH^{-1/2}(\Omega)^2}^2
       +\|r_1^i\|_{\rH^{1/2}(\T^2)}^2\big)\\
 &\leq c_m\bigg[
       \sum_{|i|\leq m}
          (E_i+\|\widetilde V^{\,i}\|_{\rH^1(\Omega)^2}^2)
       +\mathfrak g_m^2
       +\sum_{k=2}^m C_k^2\sum_{|i|\leq m-k+1}
          (E_i+\|\widetilde V^{\,i}\|_{\rH^1(\Omega)^2}^2)
       +\sum_{k=1}^m C_k^2\mathfrak g_{m-k}^2\bigg].
 \end{aligned}
\end{equation*}
The boundary sources contain only height and prescribed-source
terms, so
\begin{equation*}
 \sum_{|i|\leq m}\sum_{a=0}^1
       \|\eta_{a,i}\|_{\rH^{-1/2}(\T^2)^2}^2
 \leq c_m\sum_{k=0}^m C_k^2\bigg(
       \sum_{|i|\leq m-k}E_i+\mathfrak g_{m-k}^2\bigg).
\end{equation*}
For the viscous fluxes, the product rule gives
\begin{equation*}
 \begin{aligned}
 \sum_{|i|\leq m}\big(
       \|\Pi_\rH^i\|_{\rL^2(\Omega)^{2\times2}}^2
       +\|\Pi_z^i\|_{\rL^2(\Omega)^2}^2\big)
 &\leq c_m\sum_{k=1}^m C_k^2\sum_{|i|\leq m-k}
       \|\widetilde V^{\,i}\|_{\rH^1(\Omega)^2}^2 \ \text{ and }  \\
 \sum_{|i|\leq m}
       \|\Pi_\rH^i\|_{\rL^2((0,1),\rH^{1/2}(\T^2)^{2\times2})}^2
 &\leq c_m\sum_{k=1}^m C_k^2\sum_{|i|\leq m-k+1}
       \|\widetilde V^{\,i}\|_{\rH^1(\Omega)^2}^2.
 \end{aligned}
\end{equation*}
The second inequality uses
$\rH^1(\T^2)\hookrightarrow\rH^{1/2}(\T^2)$ on
$\nabla\widetilde V^{\,i-j}$, requiring at most one additional
horizontal derivative. The same observation controls the
differentiated integrals in $r_1^i$. 
Applying \eqref{eq:structured-averaged-inverse} with these bounds
estimates the elliptic contribution to the equation for
$\widetilde h^{\,i}$. As in \autoref{lem:basic-energy}, its
$\rH^{1/2}$ energy is obtained by multiplication of its source norm
by $\|\widetilde h^{\,i}\|_{\rH^{1/2}}$.
Next, observe that 
\begin{equation*}
 |\int_\Omega F_i\cdot\widetilde V^{\,i}\,\dd x|
 \leq\|F_i\|_{\rH^{-1/2}(\Omega)^2}
       \|\widetilde V^{\,i}\|_{\rH^{1/2}(\Omega)^2},
 \quad
 \|\widetilde V^{\,i}\|_{\rH^{1/2}(\Omega)^2}
 \leq\delta\|\widetilde V^{\,i}\|_{\rH^1(\Omega)^2}
       +C_\delta\sqrt{E_i}.
\end{equation*}
The interpolation follows from compactness and $Q$ being comparable
to $1-z$.  The flux and boundary integrals are bounded by
H\"older's inequality, the trace theorem and Young's inequality.
Testing momentum with $\widetilde V^{\,i}$ and height with
$(H/2)h^i$ cancels the pressure terms. Adding the fractional height
estimate and absorbing the small multiples of the dissipation gives
\begin{equation*}
 \begin{aligned}
 &\quad\frac{\dd}{\dd t}\sum_{|i|\leq m}E_i
       +c\sum_{|i|\leq m}
          \|\widetilde V^{\,i}\|_{\rH^1(\Omega)^2}^2\\
 &\leq c_m\bigg[
       \sum_{|i|\leq m}E_i+\mathfrak g_m^2
       +\sum_{|i|\leq m-1}
          \|\widetilde V^{\,i}\|_{\rH^1(\Omega)^2}^2
       +\sum_{k=2}^m C_k^2\sum_{|i|\leq m-k+1}
          (E_i+\|\widetilde V^{\,i}\|_{\rH^1(\Omega)^2}^2)
       +\sum_{k=1}^m C_k^2\mathfrak g_{m-k}^2\bigg].
 \end{aligned}
\end{equation*}
Starting with $E_0$, Gronwall's inequality and induction control
the lower-order terms. Retaining the displayed coefficient factors
and interpolating between orders $m_0$ and $m$ gives
$C_kC_{m-k+1}\leq c_m C_m$ and
\begin{equation*}
 C_k^2\int_0^T\mathfrak g_{m-k+1}(t)^2\,\dd t
 \leq c_m\bigg(
       \int_0^T\mathfrak g_m(t)^2\,\dd t
       +C_m^2\int_0^T\mathfrak g_{m_0}(t)^2\,\dd t\bigg) \ \text{ for } \ 2\leq k\leq m.
\end{equation*}
The terms with $\mathfrak g_{m-k}$ obey the same bound.
This proves the following estimate on the same time interval
at all orders.

\begin{lem}[Higher order tame energy estimate]\label{lem:higher-energy}
Let $(H,V)$ be a prescribed smooth background with $H(0)=H_0$, choose $T>0$ sufficiently small that $\phi(\cdot,t)$ is well defined and let $(h,\widetilde V)$
be a smooth solution of \eqref{eq:off-exact-momentum}
subject to \eqref{eq:good-boundary}, whose initial time derivatives vanish at every order.
Then for every $m\geq m_0$, we have
\begin{equation*}
 \sup_{0\leq t\leq T}\sum_{|i|\leq m}E_i(t)
 +\sum_{|i|\leq m}\int_0^T
       \|\widetilde V^{\,i}(t)\|_{\rH^1(\Omega)^2}^2\,\dd t
 \leq c_{m,T}\int_0^T\left(
       \mathfrak g_m(t)^2
       +C_m(H,V)^2\mathfrak g_{m_0}(t)^2\right)\,\dd t.
\end{equation*}
The constant $c_{m,T}$ depends on the background only through
$C_{m_0}$ and the positive chart bounds.
The integers $m_0,d$ are fixed.
\end{lem}

It remains to establish vertical regularity. 
Observe that
\autoref{lem:higher-energy} already controls one vertical derivative, so that
\begin{equation*}
 \begin{aligned}
 \|h\|_{\rH^m((0,T)\times\T^2)}^2
 +\sum_{\substack{|i|+\ell\leq m\\\ell=0,1}}
       \|\partial^i\partial_z^\ell\widetilde V
                    \|_{\rL^2((0,T)\times\Omega)^2}^2\leq c_m\int_0^T\sum_{|i|\leq m}
       \left(E_i(t)
       +\|\widetilde V^{\,i}(t)\|_{\rH^1(\Omega)^2}^2\right)\,\dd t.
 \end{aligned}
\end{equation*}
Normal differentiation also produces normal derivatives of $g_2$.
We therefore use the ordinary source norm
\begin{equation*}
 |g|_r=
 \|g_1\|_{\rH^r((0,T)\times\T^2)}
 +\|g_2\|_{\rH^r((0,T)\times\Omega)^2}
 +\sum_{a=3}^4\|g_a\|_{\rH^r((0,T)\times\T^2)^2},
\end{equation*}
so that in particular, it holds that
\begin{equation*}
 \left(\int_0^T\mathfrak g_n(t)^2\,\dd t\right)^{1/2}
       \leq c_n|g|_{n+1}.
\end{equation*}
Now, write the viscosity in divergence form as
\begin{equation*}
 \divH(J\mathcal S_HU)+\partial_z\mathcal T_HU
 =\sum_{a,b=1}^3\partial_a(A^{ab}\partial_bU).
\end{equation*}
Here $\partial_3=\partial_z$, and with $e_1,e_2$ denoting
the coordinate vectors in $\R^2$, the coefficient matrices are
\begin{equation*}
 \begin{aligned}
 A^{ab}&=\mu J\delta_{ab}I_2
       +\mu J e_b\otimes e_a+(\mu'-\mu)J e_a\otimes e_b,
                         &&a,b\in\{1,2\},\\
 A^{a3}&=-\mu(\partial_{x_a}\zeta)I_2
       -\mu\nablaH\zeta\otimes e_a
       -(\mu'-\mu)e_a\otimes\nablaH\zeta,
                         &&a\in\{1,2\},\\
 A^{3a}&=(A^{a3})^\top,
                         &&a\in\{1,2\},\\
 A^{33}&=\frac1J\left[
       \mu(1+|\nablaH\zeta|^2)I_2
       +\mu'\nablaH\zeta\otimes\nablaH\zeta\right].
 \end{aligned}
\end{equation*}
Choosing time $T>0$ small enough implies 
\begin{equation*}
 \xi\cdot A^{33}\xi\geq\frac{\mu}{J}|\xi|^2
       \geq c|\xi|^2
 \ \text{ and }\
 \|(A^{33})^{-1}\|_{\rL^\infty((0,T)\times\Omega)}
       \leq c^{-1}.
\end{equation*}
Consequently, the momentum equation determines the second
normal derivative through
\begin{equation*}
 \begin{aligned}
 A^{33}\partial_z^2\widetilde V
 =Q(\partial_t\widetilde V+\nablaH h)
       -\mathcal G(h,\widetilde V)+Q\widetilde V
 -\sum_{(a,b)\ne(3,3)}
       A^{ab}\partial_a\partial_b\widetilde V
       -\sum_{a,b=1}^3
       (\partial_aA^{ab})\partial_b\widetilde V.
 \end{aligned}
\end{equation*}
Multiplying by $(A^{33})^{-1}$ and using the
explicit formula for $\mathcal G$ gives
\begin{equation*}
 \begin{aligned}
&\quad \|\partial_z^2\widetilde V\|_{\rL^2((0,T)\times\Omega)^2}^2
 \\&\leq c_2\biggl[
       \|\sqrt Q\,\partial_t\widetilde V
                         \|_{\rL^2((0,T)\times\Omega)^2}^2
       +\|\widetilde V\|_{\rL^2(0,T,\rH^1(\Omega)^2)}^2+\sum_{a=1}^2
       \|\partial_{x_a}\widetilde V
                         \|_{\rL^2(0,T,\rH^1(\Omega)^2)}^2
       +\|h\|_{\rL^2(0,T,\rH^1(\T^2))}^2
       +|g|_0^2\biggr].
 \end{aligned}
\end{equation*}
Indeed, the horizontal and mixed second derivatives are
controlled by the norms of $\partial_{x_a}\widetilde V$ in
$\rH^1$. The remaining velocity terms contain at most one
spatial derivative, while the height terms are controlled
by $\|h\|_{\rH^1(\T^2)}$ using the harmonic-extension bounds.
The prescribed forcing is bounded by $|g|_0$.
Consequently, we obtain
\begin{equation*}
 \|\partial_z^2\widetilde V\|_{\rL^2((0,T)\times\Omega)^2}^2
 \leq c_2\bigg[
       \int_0^T\sum_{|i|\leq1}
       \left(E_i(t)
       +\|\widetilde V^{\,i}(t)\|_{\rH^1(\Omega)^2}^2\right)\,\dd t
       +|g|_0^2\bigg].
\end{equation*}
This supplies the first normal derivative missing from the
energy estimate.
For higher orders, we apply $\partial^i\partial_z^{k-2}$
to the same momentum identity, where $k\geq2$ and
$|i|+k=n$. The term to be determined is
$A^{33}\partial^i\partial_z^k\widetilde V$.
The other second-order terms contain fewer than $k$ normal
derivatives at total order $n$. When a derivative falls on
a coefficient, the velocity derivative has total order at
most $n-1$. The differentiated time term also has total
order at most $n-1$.
Thus we estimate the right-hand side first using the bounds
at lower total orders and then the bounds at the same total
order with fewer normal derivatives. Inverting $A^{33}$
at each step recovers all derivatives of total order at
most $m$.
The same induction retains tame dependence on the background.
After enlarging the fixed margin $d$ if necessary, the product
estimates give, for $n\geq2$,
\begin{equation*}
 \begin{aligned}
 &\|h\|_{\rH^n((0,T)\times\T^2)}
       +\|\widetilde V\|_{\rH^n((0,T)\times\Omega)^2}\\
 &\leq c_n\biggl[
 \left(\int_0^T\sum_{|i|\leq n}
       \left(E_i(t)+\|\widetilde V^{\,i}(t)\|_{\rH^1(\Omega)^2}^2\right)
                       \,\dd t\right)^{1/2}
       +|g|_{n-2}+\|\widetilde V\|_{\rH^{n-1}((0,T)\times\Omega)^2}\\
 &\quad+\sum_{r=1}^{n-2}C_r\left(
       \|h\|_{\rH^{n-r}((0,T)\times\T^2)}
       +\|\widetilde V\|_{\rH^{n-r}((0,T)\times\Omega)^2}
       +|g|_{n-r-2}\right)\biggr],
 \end{aligned}
\end{equation*}
where the sum is empty for $n=2$. Indeed, when $r\geq1$
of the $n-2$ differentiations fall on a second-order coefficient,
the unknown has total order at most $n-r$. The other velocity
terms have order at most $n-1$. The harmonic-extension terms
contain at most $n-r-1$ derivatives of $Eh$ and $n-r-2$
derivatives of $Eg_1$, which are controlled by the displayed norms.
Using \autoref{lem:higher-energy} and the interpolation bounds
\begin{equation*}
 C_rC_{m-r}\leq c_mC_m
 \ \text{ and }\
 C_r|g|_{m-r+1}\leq c_m\bigl(|g|_{m+1}+C_m|g|_{m_0+1}\bigr),
 \quad 1\leq r\leq m-2,
\end{equation*}
induction therefore gives
\begin{equation*}
 \|h\|_{\rH^m((0,T)\times\T^2)}
       +\|\widetilde V\|_{\rH^m((0,T)\times\Omega)^2}
 \leq c_{m,T}\bigl(|g|_{m+1}+C_m(H,V)|g|_{m_0+1}\bigr),
 \quad m\geq m_0.
\end{equation*}
Orders at most $m_0$ use only the fixed low background norm.
Finally, the product estimate for
$\delta V=\widetilde V+(Eh)V_z/J$ gives the same tame bound
for $(h,\delta V)$ after increasing the fixed margin $d$ if needed.
All estimates hold on the same time interval.
\subsection{Construction of the smooth linear inverse}
With all a priori estimates in place, it remains to establish existence of a solution $(h, \widetilde V)$ to \eqref{eq:off-exact-momentum} subject to \eqref{eq:good-boundary}. 
We fix the prescribed smooth background $(H,V)$ and choose $T>0$
so that the coordinate transformation is well defined. The sources
$(g_1,g_2,g_3,g_4)$ are smooth and vanish to every order at $t=0$.
We construct $(\widetilde h,\widetilde V)$ by successively solving
momentum with a prescribed modified height and then solving the
modified-height equation with the resulting velocity. 
For the modified height, we define the Banach space $\rZ_T$ by
\begin{equation*}
 \rZ_T=\big\{\vartheta\in\rC([0,T],\rH^{1/2}(\T^2))
                    \mid\vartheta(0)=0\big\}
\end{equation*}
subject to its canonical norm.
Given $\vartheta\in\rZ_T$, we first solve for a velocity $U$ by
substituting $h=\vartheta+B_HU$ into momentum. This gives
\begin{equation}\label{eq:picard-momentum}
 \left\{\begin{aligned}
 Q\bigl(U_t+\nablaH(\vartheta+B_HU)\bigr)+\mathcal L_HU
   &=\mathcal G(\vartheta+B_HU,U)
          &&\text{in }(0,T)\times\Omega,\\
 U(0)&=0.
 \end{aligned}\right.
\end{equation}
Its boundary conditions are
\begin{equation}\label{eq:picard-boundary}
 \mathcal T_HU|_{z=1}=G_1(\vartheta+B_HU)
 \ \text{ and }\
 U_z|_{z=0}=G_0(\vartheta+B_HU).
\end{equation}
In particular, every term depending on $U$, including the terms
containing $B_HU$, belongs to this linear velocity problem.

To construct its weak solution, we define the bilinear form
$p_t$ and the functional $\ell_\vartheta$ by
\begin{equation*}
 \begin{aligned}
 p_t(U,\psi)={}&\int_\Omega\left[
 J\operatorname{tr}\bigl((\mathcal S_HU)^\top\nablaH\psi\bigr)
       +\mathcal T_HU\cdot\partial_z\psi+QU\cdot\psi\right]\,\dd x\\
 &+\int_\Omega\left[Q\nablaH(B_HU)
             -\mathcal G(B_HU,U)+\widetilde g_2\right]\cdot\psi\,\dd x
       -\int_{\T^2}\bigl[G_1(B_HU)-g_3\bigr]
                                      \cdot\psi|_{z=1}\,\dd x_\rH\\
 &+\int_{\T^2}\frac\mu{J|_{z=0}}
       \bigl[G_0(B_HU)-g_4\bigr]\cdot\psi|_{z=0}\,\dd x_\rH,\\
 \ell_\vartheta(\psi)={}&
 \int_\Omega\left[\mathcal G(\vartheta,0)
                         -Q\nablaH\vartheta\right]\cdot\psi\,\dd x
       +\int_{\T^2}G_1(\vartheta)\cdot\psi|_{z=1}\,\dd x_\rH
       -\int_{\T^2}\frac\mu{J|_{z=0}}
                      G_0(\vartheta)\cdot\psi|_{z=0}\,\dd x_\rH.
 \end{aligned}
\end{equation*}
The subtractions in $p_t$ remove the prescribed sources, since
$\mathcal G(0,0)=\widetilde g_2$, $G_1(0)=g_3$ and $G_0(0)=g_4$.
Thus
\eqref{eq:picard-momentum} subject to \eqref{eq:picard-boundary}
has the weak formulation
\begin{equation*}
 \int_\Omega QU_t\cdot\psi\,\dd x+p_t(U,\psi)
      =\ell_\vartheta(\psi) \ \text{ for all }\ \psi\in\rH^1(\Omega)^2.
\end{equation*}
The previously established estimates of $\mathcal G$, $G_1$, $G_0$
and $B_H$ give
\begin{equation*}
 \begin{aligned}
 |p_t(U,\psi)|&\leq C\|U\|_{\rH^1(\Omega)^2}
                                  \|\psi\|_{\rH^1(\Omega)^2},\\
 \|\ell_\vartheta\|_{(\rH^1(\Omega)^2)^*}
       &\leq C\bigl(\|\vartheta\|_{\rH^{1/2}(\T^2)}
                                           +\mathfrak g(t)\bigr)
 \ \text{ and }\
 \|\ell_{\vartheta_1}-\ell_{\vartheta_2}\|_{(\rH^1(\Omega)^2)^*}
       \leq C\|\vartheta_1-\vartheta_2\|_{\rH^{1/2}(\T^2)}.
 \end{aligned}
\end{equation*}
Moreover, the first integral defining $p_t$ is the coercive
shifted viscous form defined in \eqref{eq:form} and when $\psi=U$, we obtain
\begin{equation*}
 C\|U\|_{\rH^1(\Omega)^2}\|U\|_{\rL^2(\Omega)^2}
 +C\|\sqrt Q\,U\|_{\rL^2(\Omega)^2}\|U\|_{\rH^1(\Omega)^2}
 \leq\varepsilon\|U\|_{\rH^1(\Omega)^2}^2
       +C_\varepsilon\|\sqrt Q\,U\|_{\rL^2(\Omega)^2}^2.
\end{equation*}
Consequently,
\begin{equation*}
 p_t(U,U)\geq c\|U\|_{\rH^1(\Omega)^2}^2
                        -C\|\sqrt Q\,U\|_{\rL^2(\Omega)^2}^2.
\end{equation*}

To apply Lions' variational existence argument
\cite[Chapter~IV, Theorem~1.1]{Lions1961}, see also \cite[Theorem 1.1]{ADLO}, we define
\begin{equation*}
 \rV=\rH^1(\Omega)^2 \ \text{ and } \
 \rX=\rL^2(\Omega,(1-z)\,\dd x)^2.
\end{equation*}
Moreover, set $\sigma=\sqrt{\frac Q{1-z}}
  \text{ and }\widehat U=\sigma U.$
Both $\sigma$ and $\sigma^{-1}$ are smooth and bounded.
Testing with $\sigma^{-1}\psi$ changes the bilinear form to
\begin{equation*}
 \widehat p_t(u,\psi)
 =p_t(\sigma^{-1}u,\sigma^{-1}\psi)
       -\int_\Omega(1-z)\frac{\sigma_t}{\sigma}u\cdot\psi\,\dd x.
\end{equation*}
The transformed form then satisfies
\begin{equation*}
 \begin{aligned}
 |\widehat p_t(u,\psi)|
 \leq C\|u\|_{\rH^1(\Omega)^2}
          \|\psi\|_{\rH^1(\Omega)^2} \ \text{ and } \
 \widehat p_t(u,u)
 \geq c\|u\|_{\rH^1(\Omega)^2}^2
       -C\|\sqrt{1-z}\,u\|_{\rL^2(\Omega)^2}^2,
 \end{aligned}
\end{equation*}
for all $u,\psi\in\rH^1(\Omega)^2$, where $c>0$ and
the constants are uniform for $t\in[0,T]$.
Therefore we obtain a unique solution with
\begin{equation*}
 \widehat U\in\rL^2(0,T,\rV)\cap\rC([0,T],\rX)
 \ \text{ and }\
 \partial_t\bigl((1-z)\widehat U\bigr)
                                  \in\rL^2(0,T,\rV^*).
\end{equation*}
We denote the resulting velocity by $U[\vartheta]$.
Subtracting the problems for $\vartheta_1$ and $\vartheta_2$,
and writing $U_a=U[\vartheta_a]$ for $a=1,2$, gives
\begin{equation}\label{eq:picard-velocity-difference}
 \begin{aligned}
 \sup_{0\leq t\leq T}\|\sqrt Q\,(U_1-U_2)(t)\|_{\rL^2(\Omega)^2}^2
       +\int_0^T\|U_1-U_2\|_{\rH^1(\Omega)^2}^2\,\dd t
       \leq CT\|\vartheta_1-\vartheta_2\|_{\rZ_T}^2.
 \end{aligned}
\end{equation}
We next solve the modified-height equation with this velocity.
To display this second solve, we define $\mathcal R(\eta,U)$ by
\begin{equation*}
 \begin{aligned}
 \mathcal R(\eta,U)={}&g_1+B_H\nablaH(\eta+B_HU)
       -(\partial_tB_H)U-b_1\cdot\nablaH(B_HU)-b_2B_HU\\
 &-\frac2H\divH\mathcal Q\bigl(
       Q\mathcal L_H^{-1}\mathcal G(\eta+B_HU,U)\bigr)
       -\frac2H\divH\mathcal Q\bigl(
       Q\mathcal N_H(G_1(\eta+B_HU),G_0(\eta+B_HU))\bigr).
 \end{aligned}
\end{equation*}
Thus $\mathcal R$ is exactly the right-hand side of
\eqref{eq:modified-height-equation} after $h=\eta+B_HU$.
The estimate \eqref{eq:modified-height-source-bound} gives
\begin{equation*}
 \begin{aligned}
 \|\mathcal R(0,0)\|_{\rH^{1/2}(\T^2)}&\leq C\mathfrak g(t) \ \text{ and } \\
 \|\mathcal R(\eta_1,U_1)-\mathcal R(\eta_2,U_2)\|_{\rH^{1/2}(\T^2)}
 &\leq C\bigl(\|\eta_1-\eta_2\|_{\rH^{1/2}(\T^2)}
                         +\|U_1-U_2\|_{\rH^1(\Omega)^2}\bigr).
 \end{aligned}
\end{equation*}
For the prescribed $U[\vartheta]$, we solve
\begin{equation}\label{eq:picard-height}
 \eta_t+b_1\cdot\nablaH\eta+b_2\eta
       =\mathcal R(\eta,U[\vartheta])
       \quad\text{in }(0,T)\times\T^2
 \ \text{ and }\ \eta(0)=0.
\end{equation}
Indeed, the smooth prescribed flow defined by
\begin{equation*}
 \partial_t\Phi(t,x_\rH)=b_1(t,\Phi(t,x_\rH))
 \ \text{ and }\ \Phi(0,x_\rH)=x_\rH
\end{equation*}
acts boundedly, together with its inverse, on $\rH^{1/2}(\T^2)$. Indeed, since $b_1$ is smooth, its flow $\Phi(t,\cdot)$ is a smooth
orientation-preserving diffeomorphism of $\T^2$.
Along this flow, \eqref{eq:picard-height} becomes the linear
integral equation
\begin{equation*}
 \eta(t,\Phi(t,x_\rH))
 =\int_0^t\bigl[\mathcal R(\eta,U[\vartheta])-b_2\eta\bigr]
                           (s,\Phi(s,x_\rH))\,\dd s.
\end{equation*}
Its operator on $\eta$ is bounded on $\rH^{1/2}(\T^2)$,
and its prescribed forcing belongs to
$\rL^2(0,T,\rH^{1/2}(\T^2))$. The linear integral equation
therefore has a unique solution $\eta\in\rZ_T$.

These two successive solves define the contraction map $\Psi_T$ by
\begin{equation*}
 \begin{aligned}
 \Psi_T\colon\rZ_T&\to\rZ_T, \quad
 \vartheta \mapsto U[\vartheta] \mapsto \eta.
 \end{aligned}
\end{equation*}
Indeed, if
$\eta_a=\Psi_T(\vartheta_a)$ for $a=1,2$, the transport estimate
and Gronwall's inequality give
\begin{equation*}
 \|\eta_1-\eta_2\|_{\rZ_T}
 \leq C\int_0^T\|U_1-U_2\|_{\rH^1(\Omega)^2}\,\dd t
 \leq C\sqrt T\,
               \|U_1-U_2\|_{\rL^2(0,T,\rH^1(\Omega)^2)}.
\end{equation*}
Combining this with \eqref{eq:picard-velocity-difference} yields
\begin{equation*}
 \|\Psi_T(\vartheta_1)-\Psi_T(\vartheta_2)\|_{\rZ_T}
       \leq CT\|\vartheta_1-\vartheta_2\|_{\rZ_T}.
\end{equation*}
Here $C$ is uniform for $T$ in a fixed sufficiently small interval
and depends only on finitely many background norms and the chart
bounds. Choosing $CT<1$, the contraction mapping principle yields
\begin{equation*}
 \widetilde h=\Psi_T(\widetilde h),\quad
 \widetilde V=U[\widetilde h]
 \ \text{ and }\ h=\widetilde h+B_H\widetilde V.
\end{equation*}
Explicitly, the iterates
$\vartheta^{(0)}=0$ and
$\vartheta^{(k+1)}=\Psi_T(\vartheta^{(k)})$ converge in $\rZ_T$,
and $U[\vartheta^{(k)}]$ converges in the velocity norms of
\eqref{eq:picard-velocity-difference}. The index $k$ is used only
for this construction and does not denote a Nash-Moser step.
The reconstruction of $h$ is valid for this weak solution.
To justify its time derivative, we define the weak averaged inverse
$\mathcal A_H$ by
\begin{equation*}
 \mathcal A_H\colon(\rH^1(\Omega)^2)^*\to\rL^2(\T^2),
 \quad F\mapsto\frac2H\divH\mathcal Q(Q\mathcal L_H^{-1}F)
\end{equation*}
so that
$B_HU=\mathcal A_H(QU)$. Differentiating the weak elliptic inverse
with respect to the prescribed smooth background also gives
\begin{equation*}
 \partial_t\mathcal A_H
       \in\mathcal L\bigl(\rV^*,\rL^2(\T^2)\bigr).
\end{equation*}
Since
$\partial_t(Q\widetilde V)\in\rL^2(0,T,\rV^*)$,
we obtain
\begin{equation*}
 \partial_t(B_H\widetilde V)
       =(\partial_t\mathcal A_H)(Q\widetilde V)
                    +\mathcal A_H\partial_t(Q\widetilde V).
\end{equation*}
Substitution in the modified-height equation recovers the original
height equation. Thus the fixed point solves the full linear
correction problem, with the boundary conditions in the weak sense.
The contraction mapping principle gives a unique modified height,
and uniqueness of the velocity solve gives uniqueness of
$(h,\widetilde V)$.
Moreover, for the same prescribed background, if two sets
of sources agree on $[0,t_\ast]$, their solutions also agree
on $[0,t_\ast]$. Indeed, their restrictions solve the same
problem on that interval, where the preceding uniqueness
argument applies.

We now establish smoothness of the constructed solution.
Using the prescribed flow $\Phi$ to remove the height transport
and normalizing the velocity time derivative, we work in fixed
energy spaces. The successive solves give a bounded inverse
for independently prescribed height and momentum sources.
The elliptic estimates for the inverse and boundary lifting
give smooth dependence of the resulting operators on translations
of the prescribed coefficients. Difference quotients and
uniqueness therefore show that all time and horizontal derivatives
of the solution belong to the basic energy spaces.

Recovering the normal derivatives from the momentum equation
using the uniformly invertible matrix $A^{33}$ then gives
\begin{equation*}
 h\in\bigcap_{m\geq0}\rH^m((0,T)\times\T^2)
 \ \text{ and }\
 \widetilde V\in
 \bigcap_{m\geq0}\rH^m((0,T)\times\Omega)^2.
\end{equation*}
Hence Sobolev embedding gives smoothness up to the boundaries.
The difference-quotient argument also applies across $t=0$
after extending the background smoothly and the sources by
zero to negative time. Uniqueness makes the extended solution
zero there, so all its initial derivatives vanish.
Finally, $\delta V=\widetilde V+(Eh)V_z/J$ is smooth as well.
All estimates hold on the interval fixed by the original
contraction. Applying the preceding a priori estimates therefore
gives the following result.
\begin{prop}[Smooth tame linear inverse]\label{prop:linear-inverse}
Let $(H,V)$ be a prescribed smooth background, and choose
$T>0$ sufficiently small that the coordinate transformation
is well defined and the contraction above holds.
For every smooth source $(g_1,g_2,g_3,g_4)$ vanishing to every
order at $t=0$, the full linearized problem has a unique smooth
solution $(h,\delta V)$ whose initial time derivatives all vanish.
For the fixed integers $m_0,d$ in the preceding estimates
and every $m\geq m_0$, it satisfies
\begin{equation*}
 \|h\|_{\rH^m((0,T)\times\T^2)}
 +\|\delta V\|_{\rH^m((0,T)\times\Omega)^2}
 \leq c_{m,T}\bigl(
       |g|_{m+d}+C_m(H,V)|g|_{m_0+d}\bigr).
\end{equation*}
The constant $c_{m,T}$ is uniform when the prescribed chart
bounds and a fixed low background norm are controlled.
 Moreover, if two sources
agree on $[0,t_*]$, their solutions agree on $[0,t_*]$.
\end{prop}

\section{Nonlinear existence and uniqueness}\label{sec:nonlinear-existence}
It only remains to prove the main theorem \autoref{thm:vacuum-existence} which is done below.
\begin{proof}[Proof of \autoref{thm:vacuum-existence}]
By the compatibility conditions \autoref{def:vacuum-compatibility}, we choose a smooth pair
$(H^{\mathrm{app}},V^{\mathrm{app}})$ with the prescribed
initial values and all initial time derivatives determined
by the equations and the change of variables. Thus
\begin{equation*}
 H^{\mathrm{app}}|_{t=0}=H_0
 \ \text{ and }\
 V^{\mathrm{app}}|_{t=0}=V_0.
\end{equation*}
As in \cite[Section~2]{Lindblad}, this pair serves as an
approximate solution whose residual vanishes to every
order at the initial time. More precisely, we define
$\mathfrak r$ by
\begin{equation*}
 \mathfrak r=\mathcal F(H^{\mathrm{app}},V^{\mathrm{app}})
 \ \text{ and obtain }\
 \partial_t^j\mathfrak r|_{t=0}=0
 \quad\text{for every }j\geq0.
\end{equation*}
We choose $T_0>0$ sufficiently small that the coordinate
transformation is well defined and \autoref{prop:linear-inverse}
applies. This interval is kept fixed below.

We seek smooth corrections $(h,\delta V)$ whose time derivatives
of every order vanish at $t=0$. The prescribed sources
$g=(g_1,g_2,g_3,g_4)$ satisfy the same condition.
For $m\in\N_0$, we define their Sobolev norms by
\begin{equation*}
 \begin{aligned}
 |(h,\delta V)|_m
 &=\|h\|_{\rH^m((0,T_0)\times\T^2)}
   +\|\delta V\|_{\rH^m((0,T_0)\times\Omega)^2}
 \ \text{ and }\\
 |g|_m
 &=\|g_1\|_{\rH^m((0,T_0)\times\T^2)}
   +\|g_2\|_{\rH^m((0,T_0)\times\Omega)^2}
   +\sum_{j=3}^4
       \|g_j\|_{\rH^m((0,T_0)\times\T^2)^2}.
 \end{aligned}
\end{equation*}
The smoothing and interpolation properties required in
\cite[Theorem~3.4]{Poppenberg} are then satisfied by \cite[Proposition~3.3, with $m=1$]{Poppenberg},
after a fixed shift of the derivative index.
To apply the inverse theorem at zero, we define the centered
residual $\mathcal F_{\mathrm{app}}$ by
\begin{equation*}
 \mathcal F_{\mathrm{app}}(h,\delta V)
 =\mathcal F(H^{\mathrm{app}}+h,V^{\mathrm{app}}+\delta V)
       -\mathfrak r.
\end{equation*}
This map preserves the stated initial conditions and satisfies
$\mathcal F_{\mathrm{app}}(0)=0$. Its derivative is precisely the
full linearized operator at the corrected background. The products,
positive reciprocals, harmonic extension, vertical integration and
traces in \eqref{eq:full-residual} give a smooth map. More precisely,
there is a fixed integer $d_0$ such that,
for a correction $a=(h,\delta V)$ and correction directions $a_1,a_2$,
\begin{equation*}
 \begin{aligned}
 |D\mathcal F_{\mathrm{app}}(a)a_1|_m
 &\leq c_m\left(|a_1|_{m+d_0}
       +(1+|a|_{m+d_0})|a_1|_{d_0}\right),\\
 |D^2\mathcal F_{\mathrm{app}}(a)[a_1,a_2]|_m
 &\leq c_m\left(|a_1|_{m+d_0}|a_2|_{d_0}
       +|a_2|_{m+d_0}|a_1|_{d_0}
       +(1+|a|_{m+d_0})|a_1|_{d_0}|a_2|_{d_0}\right).
 \end{aligned}
\end{equation*}
Here the constants may depend on the fixed approximation, while
$a$ ranges over the chosen neighborhood. Together with the
tame linear inverse in \autoref{prop:linear-inverse}, these bounds
verify the hypotheses of \cite[Theorem~3.4]{Poppenberg}.
Consequently, $\mathcal F_{\mathrm{app}}$ has a smooth local inverse
between neighborhoods of zero in these spaces.
Following \cite[Section~2]{Lindblad}, we extend $\mathfrak r$ by zero
to negative times and define its delayed value by
\begin{equation*}
 \mathfrak r_\tau(t)=\mathfrak r(t-\tau),\quad
 |\mathfrak r_\tau-\mathfrak r|_m
       \leq\tau|\mathfrak r|_{m+1}
 \ \text{ and }\
 \mathfrak r_\tau|_{[0,\tau]}=0,
 \quad 0<\tau<T_0.
\end{equation*}
The zero extension is smooth. The estimate follows by integrating its time derivative over an
interval of length $\tau$, so the difference tends to zero in every
fixed Sobolev norm. For sufficiently small $\tau$, the local inverse
therefore gives a smooth pair $(h_\tau,\delta V_\tau)$ with
\begin{equation*}
 \mathcal F_{\mathrm{app}}(h_\tau,\delta V_\tau)
       =\mathfrak r_\tau-\mathfrak r
 \ \text{ and hence }\
 \mathcal F(H^{\mathrm{app}}+h_\tau,
             V^{\mathrm{app}}+\delta V_\tau)
       =\mathfrak r_\tau.
\end{equation*}
Restricting to $[0,\tau]$ yields an exact solution with the prescribed
initial values. The inverse transformation gives $v=V\circ\phi^{-1}$,
and \eqref{eq:physical-diagnostic} determines $(H-z)w$.
The equivalence of the transformed and physical equations gives
\eqref{eq:physical-reduced} subject to
\eqref{eq:physical-reduced-boundary}. Smoothness holds up to the
moving boundary, and continuity allows us to shorten the interval
so that $H\geq\tfrac12\min_{\T^2}H_0$.

For uniqueness, shrink the range neighborhood of
$\mathcal F_{\mathrm{app}}^{-1}$ to a convex neighborhood. Its
derivative at a source is the inverse of the full linearization
at the corresponding pair. By \autoref{prop:linear-inverse}, this
derivative maps sources vanishing on $[0,t_*]$ to corrections
vanishing there. Integrating the derivative along the segment
between two sources consequently shows that their inverse images
agree on $[0,t_*]$ whenever the sources agree there.

Any two smooth solutions with the same initial data have the same
initial derivatives by the compatibility recursion. Their
differences from $(H^{\mathrm{app}},V^{\mathrm{app}})$ therefore
vanish to every order at $t=0$. Choose a smooth cutoff
$\vartheta$ equal to one on $[0,1]$ and zero on $[2,\infty)$.
For either difference $a$, we have
\begin{equation*}
 |\vartheta(t/\tau)a|_m\leq c_{m,N}\tau^N
 \quad\text{for every fixed }m,N\geq0.
\end{equation*}
For small $\tau$, with $2\tau$ inside their common interval,
these cutoff differences belong to the inverse neighborhood and
are extended by zero beyond $2\tau$. Their centered residuals
both equal $-\mathfrak r$ on $[0,\tau]$. The preceding property of
the inverse proves equality of the two solutions there.
Restarting this argument at an endpoint of agreement proves
uniqueness on their common interval of smooth existence.
\end{proof}

\end{document}